%% file: main.tex
\documentclass[10pt,%
]{amsart}
\usepackage[T1]{fontenc}
\usepackage{lmodern}
\usepackage[british]{babel}
\usepackage{amsmath, amsthm, amsfonts, amssymb}
\usepackage[all, cmtip]{xy}
\usepackage[dvips]{color}
\usepackage[dvips]{graphicx}
\usepackage{hyperref}
\usepackage{caption}
\usepackage{multirow}
\theoremstyle{definition}
\newtheorem*{ack}{Acknowledgements}
\newtheorem{defn}{Definition}[section]

\newtheorem{rmk}[defn]{Remark}
\theoremstyle{plain}

\newtheorem{cor}[defn]{Corollary}
\newtheorem{lem}[defn]{Lemma}
\newtheorem{prop}[defn]{Proposition}
\newtheorem{thm}[defn]{Theorem}
\numberwithin{equation}{section}
\numberwithin{figure}{section}
\numberwithin{table}{section}
\makeatletter
\g@addto@macro\@floatboxreset\centering
\makeatother
\DeclareMathOperator{\id}{Id}
\newcommand{\set}[1]{\left\lbrace #1 \right\rbrace}
\newcommand{\setc}[2]{
  \left\lbrace #1 \, \middle\vert \, #2 \right\rbrace
}
\newcommand{\brac}[1]{\left( #1 \right)}
\newcommand{\comp}{\DOTSB\circ}

\newcommand{\vphi}{\varphi}
\renewcommand{\phi}{\vphi}
\newcommand{\eset}{\emptyset}
\newcommand{\numset}[1]{\mathbb{#1}}

\newcommand{\Z}{\numset{Z}}

\newcommand{\R}{\numset{R}}

\newcommand{\F}[1]{\numset{F}_{#1}}
\newcommand{\cycgrp}[1]{\Z / #1 \Z}
\newcommand{\torus}{\mathbb{T}}
\newcommand{\bound}{\partial}
\DeclareMathOperator{\Rect}{Rect}
\DeclareMathOperator{\Tri}{Tri}
\DeclareMathOperator{\Quad}{Quad}
\DeclareMathOperator{\Pent}{Pent}
\DeclareMathOperator{\Hex}{Hex}
\DeclareMathOperator{\Hept}{Hept}
\DeclareMathOperator{\Int}{Int}
\DeclareMathOperator{\Hom}{H}
\DeclareMathOperator{\Kh}{Kh}
\DeclareMathOperator{\HF}{HF}
\DeclareMathOperator{\HK}{GH} 
\DeclareMathOperator{\HFK}{HFK}
\DeclareMathOperator{\HFG}{GH} 
\DeclareMathOperator{\GC}{GC} 
\DeclareMathOperator{\sgn}{sgn}
\DeclareMathOperator{\RL}{RL}
\DeclareMathOperator{\TL}{TL}
\DeclareMathOperator{\CR}{CR}
\DeclareMathOperator{\BCR}{BCR}
\newcommand{\ophat}[1]{\widehat{#1}}
\newcommand{\optilde}[1]{\widetilde{#1}}
\newcommand{\emptypoly}[1]{#1^\circ}
\newcommand{\Kht}{\optilde{\Kh}}
\newcommand{\HFh}{\ophat{\HF}}
\newcommand{\HKh}{\ophat{\HK}}
\newcommand{\HKt}{\optilde{\HK}}
\newcommand{\HFKh}{\ophat{\HFK}}
\newcommand{\HFGh}{\ophat{\HFG}}
\newcommand{\HFGt}{\optilde{\HFG}}
\newcommand{\GCt}{\optilde{\GC}}
\newcommand{\eRect}{\emptypoly{\Rect}}
\newcommand{\ePent}{\emptypoly{\Pent}}
\newcommand{\eHex}{\emptypoly{\Hex}}
\newcommand{\eHept}{\emptypoly{\Hept}}
\newcommand{\gen}[1]{\mathbf{#1}}
\newcommand{\poly}[1]{\mathcal{#1}}
\newcommand{\markers}{\mathbb{X}}
\newcommand{\markersO}{\mathbb{O}}
\newcommand{\sign}{\mathcal{S}}
\newcommand{\quasi}{\mathcal{Q}}
\newcommand{\countJ}{\mathcal{J}}
\newcommand{\grid}{\mathbb{G}}
\title[Grid diagrams and Manolescu's unoriented skein exact triangle]{Grid 
  diagrams and Manolescu's unoriented skein exact triangle for knot Floer 
  homology}
\author{C.-M. Michael Wong}
\address{%
  Department of Mathematics, Columbia University, New York, NY 10027, USA
}
\email{\href{mailto:cmmwong@math.columbia.edu}{cmmwong@math.columbia.edu}}
\urladdr{\url{http://math.columbia.edu/~cmmwong/}}
\thanks{The author was partially supported by the Princeton University 
  Mathematics Department.}
\subjclass[2010]{57R58 (Primary); 57M25, 57M27 (Secondary)}
\begin{document}
%
%
\input{abstract}
\maketitle
\input{sec_intro}

\input{sec_grid}

\input{sec_skein}

\input{sec_signs}

\input{sec_iterate}

\input{sec_quasi}

\bibliography{reference}{}
\bibliographystyle{amsalpha}
%
%
\end{document}

%% file: abstract.tex
\begin{abstract}
  We re-derive Manolescu's unoriented skein exact triangle for knot Floer 
  homology over $\F{2}$ combinatorially using grid diagrams, and extend it to 
  the case with $\Z$ coefficients by sign refinements. Iteration of the 
  triangle gives a cube of resolutions that converges to the knot Floer 
  homology of an oriented link. Finally, we re-establish the homological 
  $\sigma$-thinness of quasi-alternating links.
\end{abstract}

%% file: sec_intro.tex
\section{Introduction}

\emph{Heegaard Floer homology} is first introduced in \cite{OS04b} as an 
invariant for $3$-manifolds, defined using holomorphic disks and Heegaard 
diagrams. It is extended in \cite{OS04a, Ras03} to give an invariant, 
\emph{knot Floer homology}, for null-homologous knots in a closed, oriented 
$3$-manifold, which is further generalised in \cite{OS08} to the case of 
oriented links. Knot Floer homology comes in several flavours; its most usual 
form, $\HFKh (L)$ for an oriented link $L$, is a bi-graded module over $\F{2} = 
\Z / 2 \Z$ or $\Z$, whose Euler characteristic is the Alexander polynomial. For 
the purposes of this paper, we shall only consider links in $S^3$.

In \cite{MOS09, MOST07}, a combinatorial description of knot Floer homology 
over $\F{2}$ is given using \emph{grid diagrams}, which are certain 
multi-pointed Heegaard diagrams on the torus. In this approach, one can 
associate a chain complex $\GCt (\grid)$ to a grid diagram $\grid$, and calculate its 
homology $\HKt (\grid)$. Sign refinements for the boundary map $\bound$ are also 
given in \cite{MOST07} in a well-defined manner, allowing the chain complex to 
be defined over $\Z$. If $\grid$ is a grid diagram for a link $L$ of $\ell$ 
components, with \emph{grid number} $n$, then
\[
  \HKt (\grid) \cong \HFGh (L) \otimes V^{n-l},
\]
where $V$ is a free module of rank $2$ over the base ring $R = \F{2}$ or $\Z$, 
and $\HFGh (L)$ is a link invariant, called the \emph{combinatorial knot Floer 
homology} or the \emph{grid homology} of $L$. Over $\F{2}$, $\HFGh (L)$ is 
isomorphic to $\HFKh (L)$; over $\Z$, it has been shown in \cite{Sar11} that 
$\HFGh (L)$ is isomorphic to $\HFKh (L, \mathfrak{o})$ for some orientation 
system $\mathfrak{o}$.

Ozsv\'ath and Szab\'o observe in \cite{OS05} that, like Khovanov homology $\Kht 
(L)$ \cite{Kho00, Kho02, BN02}, Heegaard Floer homology of the branched double 
cover $\HFh (- \Sigma (L))$ satisfies an unoriented skein exact triangle.  
Manolescu then shows in \cite{Man07} that over $\F{2}$, the knot Floer homology 
also satisfies an unoriented skein exact triangle. More precisely, let 
$L_\infty$ be a link in $S^3$. Given a planar diagram of $L_\infty$, let $L_0$ 
and $L_1$ be the two resolutions of $L_\infty$ at a crossing in that diagram, 
as in Figure~\ref{fig:links}. Denote by $\ell_\infty, \ell_0, \ell_1$ the 
number of components of the links $L_\infty, L_0$, and $L_1$, respectively, and 
set $m = \max \set{\ell_\infty, \ell_0, \ell_1}$.

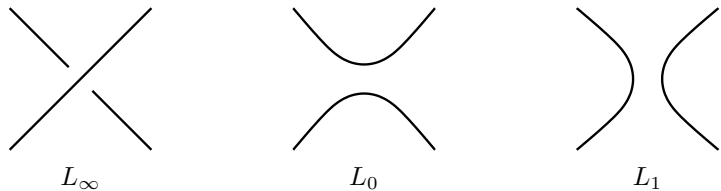
\begin{figure}
  \resizebox{0.75\textwidth}{!}{\input{fig_links.pstex_t}}
  \caption{$L_\infty, L_0$ and $L_1$ near a point.}
  \label{fig:links}
\end{figure}

\begin{thm}[Manolescu]
  \label{thm:manolescu}
  There exists an exact triangle
  \begin{gather*}
    \dotsb \to \HFKh (L_\infty; \F{2}) \otimes V^{m-\ell_\infty} \to \HFKh 
    (L_0; \F{2}) \otimes V^{m-\ell_0}\\
    \to \HFKh (L_1; \F{2}) \otimes V^{m-\ell_1} \to \dotsb,
  \end{gather*}
  where $V$ is a vector space of dimension $2$ over $\F{2}$.
\end{thm}

\begin{rmk}
  The arrows in the exact triangle point in the reverse direction from those in 
  \cite{Man07}; this is caused by a difference in the orientation convention.  
  We follow the convention in \cite{OS04a} and \cite{MOS09, MOST07}, where the 
  Heegaard surface is the oriented boundary of the handlebody in which the 
  $\alpha$ curves bound discs.
\end{rmk}

Manolescu observes that the exact triangle above is different from that of 
$\Kht (L)$ and $\HFh (- \Sigma (L))$ in that it does not even respect the 
homological grading modulo $2$, and that it is unclear whether an analogous 
triangle holds for other versions of knot Floer homology. He also uses the 
exact triangle to show that $\operatorname{rk} \HFKh (L; \F{2}) = 2^{\ell-1} 
\det (L)$ for quasi-alternating links, which explains the fact that $\Kht$ and 
$\HFKh$ have equal ranks for many classes of knots.

The goal of the present paper is to re-prove Manolescu's theorem in elementary 
terms using grid diagrams, without appealing to the topological theory. The 
advantages of this approach are three-fold.

First, Manolescu's skein exact triangle is proven in \cite{Man07} to exist over 
$\F{2}$. By assigning signs to the maps between chain complexes, we can obtain 
an analogous exact triangle in combinatorial knot Floer homology with $\Z$ 
coefficients, which has not been known to exist before. In other words, we 
obtain the following statement.

\begin{thm}
  \label{thm:main}
  For sufficiently large $n$, there exists an 
  exact triangle
  \[
    \dotsb \to \HFGh (L_\infty; R) \otimes V^{n-\ell_\infty} \to \HFGh (L_0; R) 
    \otimes V^{n-\ell_0} \to \HFGh (L_1; R) \otimes V^{n-\ell_1} \to \dotsb,
  \]
  where $R = \F{2}$ or $\Z$, and $V$ is a free module of rank $2$ over $R$.
\end{thm}

Second, the exact triangle is iterated by Baldwin and Levine in \cite{BL12} to 
obtain a cube of resolutions; when using twisted coefficients, this gives a 
combinatorial description of knot Floer homology, distinct from that provided 
by grid diagrams. In the present context, knowing explicitly the maps between 
the chain complexes associated to the grid diagrams, we can likewise iterate 
the exact triangle to get a cube of resolutions complex $\CR (\grid)$ over 
$\F{2}$, with untwisted coefficients. The higher terms in the resulting 
spectral sequence are combinatorially computable.

\begin{cor}
  \label{cor:main}
  The cube of resolutions $\CR (\grid; \F{2})$, which has no diagonal maps, gives rise to a spectral 
  sequence that converges to $\HFGh (L; \F{2}) \otimes V^{m - \ell}$.
\end{cor}

It should be noted here that the spectral sequence is presumably not a knot 
invariant; see Remark~7.7 of \cite{BL12}.

The technique of spectral sequences is first used by Ozsv\'ath and Szab\'o in 
\cite{OS05}, where a spectral sequence from $\Kht (L)$ to $\HFh (- \Sigma (L))$ 
is shown to exist. More recently, Lipshitz, Ozsv\'ath and Thurston have found a 
way in \cite{LOT14} to compute the higher terms in this spectral sequence using 
bordered Floer homology. Inspired by this work, Baldwin has found another 
method in \cite{Bal11} of computing these higher terms.

Third, there exists a \emph{$\delta$-grading} on $\HFKh$; in \cite{MO08}, 
Manolescu and Ozsv\'{a}th investigate the $\delta$-grading changes in the skein 
exact triangle \cite[Proposition~3.9]{MO08}. This result allows them to apply 
the skein exact triangle to \emph{quasi-alternating links}, to show that such 
links are \emph{Floer-homologically $\sigma$-thin} over $\F{2}$.  The 
$\delta$-gradings can also be determined in the combinatorial picture; doing 
so, we prove a generalisation of the statement of Floer-homological 
$\sigma$-thinness to $\Z$.

\begin{thm}
  \label{thm:delta_sigma}
  Suppose that $\det (L_0), \det (L_1) > 0$ and $\det (L_\infty) = \det (L_0) + 
  \det (L_1)$. Then with respect to the $\delta$-grading, the exact sequence in 
  Theorem~\ref{thm:main} can be written as
  \begin{align*}
    \dotsb & \to \HFGh_{*-\frac{\sigma(L_1)}{2}} (L_1; R) \otimes V^{n-\ell_1} 
    \to \HFGh_{*-\frac{\sigma(L_\infty)}{2}} (L_\infty; R) \otimes 
    V^{n-\ell_\infty}\\
    & \to \HFGh_{*-\frac{\sigma(L_0)}{2}}  (L_0; R) \otimes V^{n-\ell_0} \to 
    \HFGh_{*-\frac{\sigma(L_1)}{2}+1} (L_1; R) \otimes V^{n-\ell_1} \to \dotsb,
  \end{align*}
  where $R = \F{2}$ or $\Z$, and $V$ is a free module of rank $2$ over $R$ with 
  grading zero.
\end{thm}

\begin{thm}
  \label{thm:thinZ}
  Quasi-alternating links are Floer-homologically $\sigma$-thin over $\Z$.
\end{thm}

This paper is organised as follows. We review the definition of knot Floer 
homology in terms of grid diagrams in Section~\ref{sec:grid}, and re-prove the 
skein exact triangle in Section~\ref{sec:skein}. In these two sections, we will 
work only over $\F{2}$. Sign-refinements are then given in 
Section~\ref{sec:signs} to establish the analogous result over $\Z$. Next, we 
discuss how the exact triangle can be iterated to obtain a cube of resolutions 
over $\F{2}$ in Section~\ref{sec:iterate}. Finally, we establish the 
homological $\sigma$-thinness of quasi-alternating links over $\F{2}$ and $\Z$ 
in Section~\ref{sec:quasi}.

\begin{ack}
  The author is very grateful to John Baldwin for his suggestion of the topic 
  of the paper, and is indebted to John for many essential ideas in the text.  
  He also thanks John Baldwin, Robert Lipshitz, and Peter Ozsv\'ath for their 
  guidance. He thanks Ciprian Manolescu for a helpful conversation, and Gahye 
  Jeong for pointing out a previously missing case in the proof of 
  Lemma~\ref{lem:cond_1}. Lastly, he thanks the referee for remarkably thorough 
  and useful comments, and for pointing out a mistake in the statement of the 
  main theorem (Theorem~\ref{thm:main}) in an earlier version.
\end{ack}

%% file: fig_links.pstex_t
\begin{picture}(0,0)%
\includegraphics{fig_links.pstex}%
\end{picture}%
\setlength{\unitlength}{1973sp}%
\begingroup\makeatletter\ifx\SetFigFont\undefined%
\gdef\SetFigFont#1#2#3#4#5{%
  \reset@font\fontsize{#1}{#2pt}%
  \fontfamily{#3}\fontseries{#4}\fontshape{#5}%
  \selectfont}%
\fi\endgroup%
\begin{picture}(9066,2347)(1168,-2675)
\put(9301,-2611){\makebox(0,0)[b]{\smash{{\SetFigFont{10}{12.0}{\familydefault}{\mddefault}{\updefault}{\color[rgb]{0,0,0}$L_1$}%
}}}}
\put(5701,-2611){\makebox(0,0)[b]{\smash{{\SetFigFont{10}{12.0}{\familydefault}{\mddefault}{\updefault}{\color[rgb]{0,0,0}$L_0$}%
}}}}
\put(2101,-2611){\makebox(0,0)[b]{\smash{{\SetFigFont{10}{12.0}{\familydefault}{\mddefault}{\updefault}{\color[rgb]{0,0,0}$L_\infty$}%
}}}}
\end{picture}%

%% file: sec_grid.tex
\section{Grid diagrams}
\label{sec:grid}

We review the combinatorial description of knot Floer homology in terms of grid 
diagrams. In this and the next section, we will work only over $\F{2} = \Z / 2 
\Z$.

A \emph{planar grid diagram} $\Tilde{\grid}$ with \emph{grid number} $n$ is a 
square grid in $\R^2$ with $n \times n$ cells, together with a collection of 
$O$'s and $X$'s, such that
\begin{itemize}
  \item Each row contains exactly one $O$ and exactly one $X$;
  \item Each column contains exactly one $O$ and exactly one $X$; and
  \item Each cell is either empty, contains one $O$, or contains one $X$.
\end{itemize}
Given a planar grid diagram $\Tilde{\grid}$, we can place it in a standard 
position on $\R^2$ as follows: We place the bottom left corner at the origin, 
and require that each cell be a square of edge length one. We can then 
construct an oriented, planar link projection by drawing horizontal segments 
from the $O$'s to the $X$'s in each row, and vertical segments from the $X$'s 
to the $O$'s in each column. At every intersection point, we let the horizontal 
segment be the underpass and the vertical one the overpass. This gives a planar 
projection of an oriented link $L$ onto $\R^2$; we say that $\Tilde{\grid}$ is 
a \emph{planar grid presentation} of $L$.

We transfer our grid diagram to the torus $\torus$, by gluing the topmost 
segment to the bottommost one, and gluing the leftmost segment to the rightmost 
one. Then the horizontal and vertical arcs become horizontal and vertical 
circles. The torus inherits its orientation from the plane. The resulting 
diagram $\grid$ is a \emph{toroidal grid digram}, or simply a \emph{grid 
  diagram}. $\grid$ is then a \emph{grid presentation} of $L$; we also say that 
$\grid$ is a grid digram for $L$.

Given a toroidal grid diagram $\grid$, we associate to it a chain complex 
$(\GCt (\grid), \bound)$ as follows. The set of generators of $\GCt (\grid)$, 
denoted $\gen{S} (\grid)$, consists of one-to-one correspondences between the 
horizontal circles and vertical circles. Equivalently, we can regard the 
generators as $n$-tuples of intersection points between the horizontal and 
vertical circles, such that no intersection point appears on more than one 
horizontal or vertical circle.

We now define the differential map $\bound \colon \GCt (\grid) \to \GCt 
(\grid)$. Given $\gen{x}, \gen{y} \in \gen{S} (\grid)$, let $\Rect (\gen{x}, 
\gen{y})$ denote the space of embedded rectangles with the following 
properties. First of all, $\Rect (\gen{x}, \gen{y})$ is empty unless $\gen{x}, 
\gen{y}$ coincide at exactly $n - 2$ points. An element $r$ of $\Rect (\gen{x}, 
\gen{y})$ is an embedded disk in $\torus$, whose boundary consists of four 
arcs, each contained in horizontal or vertical circles; under the orientation 
induced on the boundary of $r$, the horizontal arcs are oriented from a point 
in $\gen{x}$ to a point in $\gen{y}$. The space of empty rectangles $r \in 
\Rect (\gen{x}, \gen{y})$ with $\gen{x} \cap \Int (r) = \eset$ is denoted 
$\eRect (\gen{x}, \gen{y})$.

More generally, a \emph{path} from $\gen{x}$ to $\gen{y}$ is a $1$-cycle 
$\gamma$ on $\torus$ such that the boundary of the intersection of $\gamma$ 
with the union of the horizontal curves is $\gen{y} - \gen{x}$, and a 
\emph{domain} $p$ from $\gen{x}$ to $\gen{y}$ is a two-chain in $\torus$ whose 
boundary $\bound p$ is a path from $\gen{x}$ to $\gen{y}$; the set of domains 
from $\gen{x}$ to $\gen{y}$ is denoted $\pi (\gen{x}, \gen{y})$.

Given $\gen{x} \in \gen{S} (\grid)$, we define
\[
  \bound (\gen{x}) = \sum_{\gen{y} \in \gen{S} (\grid)} \sum_{r \in \eRect 
    (\gen{x}, \gen{y})}
  \begin{cases}
    1 & \text{if $\Int (r)$ contains no $O$'s or $X$'s}\\
    0 & \text{otherwise}
  \end{cases}
  \Bigg\} \cdot \gen{y} \in \GCt (\grid).
\]

It is not too difficult to see that indeed $\bound \comp \bound = 0$, and so 
$\bound$ is a differential: We have
\[
  \bound \comp \bound (\gen{x}) = \sum_{\gen{y} \in \gen{S} (\grid)} \sum_{p 
    \in \pi (\gen{x}, \gen{z})} N (p) \cdot \gen{z},
\]
where $N (p)$ is the number of ways of decomposing $p$ as a composite of two 
empty rectangles $p = r_1 * r_2$, where $r_1 \in \eRect (\gen{x}, \gen{y})$ and 
$r_2 \in \eRect (\gen{y}, \gen{z})$. Let $p = r_1 * r_2$; then $r_1$ and $r_2$ 
either are disjoint, have overlapping interiors, or share a corner. If $r_1$ 
and $r_2$ are disjoint or have overlapping interiors, then $p = r_2 * r_1$; if 
they share a corner, then there exists a unique alternate decomposition of $p = 
r_1' * r_2'$. In any case, we obtain that $N (p) = 0$ for all $p \in \pi 
(\gen{x}, \gen{z})$.

Moreover, to the complex $\GCt (\grid)$ we can associate a \emph{Maslov 
  grading} and an \emph{Alexander grading}, determined by the functions $M 
\colon \gen{S} \to \Z$ and $S \colon \gen{S} \to \frac{1}{2} \Z$. For reasons 
we will see in the next section, we will in general not be concerned with these 
gradings, unless otherwise specified. We postpone their definitions to 
Section~\ref{sec:quasi}.  It can be checked, however, that the differential 
$\bound$ decreases the Maslov grading by $1$ and preserves the Alexander 
grading.

We can now take the homology of the chain complex $(\GCt (\grid), \bound)$, and 
define
\[
  \HKt (\grid) = \Hom_* (\GCt (\grid), \bound).
\]
It is shown in \cite{MOS09, MOST07} that, if $\grid$ is a grid diagram with 
grid number $n$ for the oriented link $L$ with $\ell$ components, then
\[
  \HKt (\grid) \cong \HKh (L) \otimes V^{n - \ell},
\]
where $V$ is a $2$-dimensional vector space over $\F{2}$, spanned by one 
generator in (Maslov and Alexander) bigrading $(-1, -1)$ and another in 
bigrading $(0, 0)$, and $\HKh (L)$ is a link invariant, often referred to as 
the \emph{combinatorial knot Floer homology}, or \emph{grid homology}, that is 
a vector space isomorphic to $\HFKh (L)$. $\HKt (\grid)$ can also be denoted by 
$\HFGt (L, n)$.

\begin{rmk}
  While the proof that $\bound$ is a differential is completely elementary, its 
  method is very useful to what we shall prove in this paper. In general, in 
  order to prove that a map defined by counting certain domains is a chain map, 
  or to prove that it is a chain homotopy, we juxtapose two domains and 
  enumerate all possible outcomes.
\end{rmk}

%% file: sec_skein.tex
\section{Manolescu's unoriented skein exact triangle}
\label{sec:skein}

We now prove our main result over $\F{2}$ in purely combinatorial terms.

Before we start our main discussion, we make a change in our notation. In the 
original description developed in \cite{MOS09, MOST07}, the $O$'s and $X$'s are 
used to determine the Alexander and Maslov gradings of the generators. However, 
given a link $L_\infty$ in $S^3$ and its two resolutions $L_0, L_1$ at a 
crossing, we observe that $L_\infty, L_0, L_1$ do not share a compatible 
orientation. Since we shall soon combine all three grid diagrams into one, we 
must forget the orientations of the links; this implies that we must also 
forget the distinctions between the $O$'s and $X$'s, and ignore the gradings. 
Notice that using only one type of markers will not change the definition of 
$\HFGt (L, n)$, since without the gradings, the definition of the chain complex 
$\GCt (\grid)$ associated to a grid diagram $G$ is symmetric in the $O$'s and 
the $X$'s. Therefore, we shall henceforth replace all $O$'s with $X$'s. 

With this new notation, the first two conditions for a grid diagram become the 
condition that there are exactly two $X$'s on each row and each column. We 
denote by $\markers$ the set of $X$'s on a grid diagram. The differential is 
then given by
\[
  \bound (\gen{x}) = \sum_{\gen{y} \in \gen{S} (\grid)} \sum_{\substack{r \in 
      \Rect^\circ (\gen{x}, \gen{y})\\ \Int (r) \cap \markers = \eset}} \gen{y} 
  \in \GCt (\grid).
\]

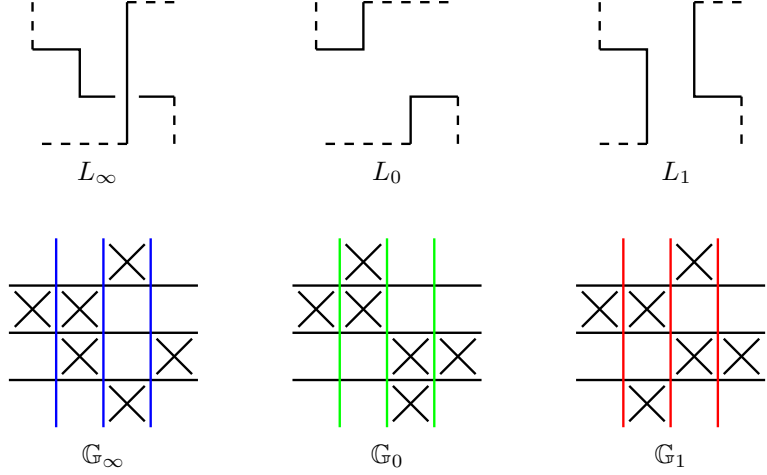
\begin{figure}
  \resizebox{0.8\textwidth}{!}{\input{fig_grid_alt.pstex_t}}
  \caption{Grid diagrams for $L_\infty, L_0$ and $L_1$ near a point.}
  \label{fig:grid}
\end{figure}

To begin, we position $L_\infty, L_0, L_1$ as in Figure~\ref{fig:grid}, and 
make sure that their respective grid diagrams $\grid_{\infty}, \grid_0, 
\grid_1$ are identical except near the crossing, as indicated in the same 
figure. Next, we let $C_k = \GCt (\grid_k)$ be the chain complex associated 
with $\grid_k$, for each $k \in \set{\infty, 0, 1}$. We endow the set 
$\set{\infty, 0, 1}$ with an action by $\cycgrp{3}$ by identifying $\infty$ 
with $2$, so that $\infty + 1 = 0$ and $1 + 1 = \infty$.

\begin{rmk}
  The leftmost and rightmost columns in each of the diagrams in 
  Figure~\ref{fig:grid} do not need to be in the form displayed; the markers 
  can be in any row in those two columns, and the proofs in this and the next 
  section do not rely on the positions of the markers. In the figures in this 
  section, we leave the markers there for ease of visualisation. However, in 
  Section~\ref{sec:iterate}, it will be necessary to have the markers exactly 
  as they appear in Figure~\ref{fig:grid}, to iterate the exact triangle.
\end{rmk}

Instead of drawing three different diagrams, we can draw $\grid_{\infty}, 
\grid_0, \grid_1$ all on the same diagram, as in Figure~\ref{fig:one_grid}. We 
label by $\beta_k$ the vertical circle corresponding to $\grid_k$, for each $k 
\in \set{\infty, 0, 1}$, as indicated. These three vertical circles, together 
with all the other vertical circles, divide $\torus$ into a number of 
components; we let $b$ be the unique component that is an annulus not 
containing any $X$ in its interior.  Also, exactly three of these components 
are embedded triangles not containing any $X$ in their interior; we let $t_k$ 
be the triangle whose $\beta_k$ arc is disjoint from the boundary of $b$.  
Finally, $\beta_k$ and $\beta_{k+1}$ intersect at exactly two points; we denote 
by $u_k$ the intersection point that lies on the boundary of $b$, and  by $v_k$ 
the other intersection point. Then $u_k = b \cap t_{k+2}$ and $v_k = t_k \cap 
t_{k+1}$.

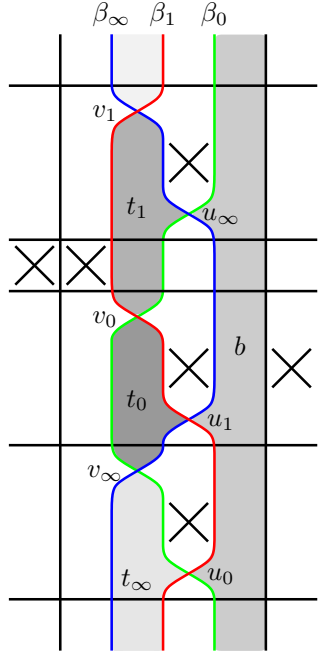
\begin{figure}
  \resizebox{0.33\textwidth}{!}{\input{fig_one_grid.pstex_t}}
  \caption{Combined grid diagram for $L_\infty, L_0$ and $L_1$ near a point. 
    The annulus $b$ and the triangles $t_k$ are shaded. The circles $\beta_k$ 
    and $\beta_{k+1}$ intersect at two points; $u_k$ are the ones on the right, 
    and $v_k$ are those on the left.}
  \label{fig:one_grid}
\end{figure}

In this setting and over $\F{2}$, Theorem~\ref{thm:main} follows from our main 
proposition:

\begin{prop}
  \label{prop:exact}
  There exists an exact triangle
  \[
    \dotsb \to \HKt (\grid_{\infty}; \F{2}) \to \HKt (\grid_0; \F{2}) \to \HKt 
    (\grid_1; \F{2}) \to \dotsb
  \]
\end{prop}

We make use of the following lemma from homological algebra \cite{OS05}, used 
also in \cite{Man07}:

\begin{lem}
  \label{lem:hom_alg}
  Let $\set{(C_k, \bound_k)}_{k \in \set{\infty, 0, 1}}$ be a collection of 
  chain complexes over an arbitrary commutative ring, and let $\set{f_k \colon 
    C_k \to C_{k+1}}_{k \in \set{\infty, 0, 1}}$ be a collection of anti-chain 
  maps such that the following conditions are satisfied for each $k$:
  \begin{enumerate}
    \item The composite $f_{k+1} \comp f_k \colon C_k \to C_{k+2}$ is 
      chain-homotopic to zero, by a chain homotopy $\phi_k \colon C_k \to 
      C_{k+2}$:
      \[
        f_{k+1} \comp f_k + \bound_{k+2} \comp \phi_k + \phi_k \comp \bound_k = 
        0;
      \]
    \item The map $\phi_{k+1} \comp f_k + f_{k+2} \comp \phi_k \colon C_k \to 
      C_k$ is a quasi-isomorphism. (In particular, if there exists a chain 
      homotopy $\psi_k \colon C_k \to C_k$ such that
      \[
        \phi_{k+1} \comp f_k + f_{k+2} \comp \phi_k + \bound_k \comp \psi_k + 
        \psi_k \comp \bound_k = \id,
      \]
      then this condition is satisfied.)
  \end{enumerate}
  Then the sequence
  \[
    \dotsb \to \Hom_* (C_k) \xrightarrow{(f_k)_*} \Hom_* (C_{k+1}) 
    \xrightarrow{(f_{k+1})_*} \Hom_* (C_{k+2}) \to \dotsb
  \]
  is exact.
\end{lem}

We now define the chain maps $f_k \colon C_k \to C_{k+1}$ by counting pentagons 
and triangles.

Given $\gen{x} \in \gen{S} (\grid_k)$ and $\gen{y} \in \gen{S} (\grid_{k+1})$, 
let $\Pent_k (\gen{x}, \gen{y})$ denote the space of embedded pentagons with 
the following properties. First of all, $\Pent_k (\gen{x}, \gen{y})$ is empty 
unless $\gen{x}, \gen{y}$ coincide at exactly $n - 2$ points. An element $p$ of 
$\Pent_k (\gen{x}, \gen{y})$ is an embedded disk in $\torus$, whose boundary 
consists of five arcs, each contained in horizontal or vertical circles; under 
the orientation induced on the boundary of $p$, we start at the 
$\beta_k$-component of $\gen{x}$, traverse the arc of a horizontal circle, meet 
its corresponding component of $\gen{y}$, proceed to an arc of a vertical 
circle, meet the corresponding component of $\gen{x}$, continue through another 
horizontal circle, meet the component of $\gen{y}$ contained in $\beta_{k+1}$, 
proceed to an arc in $\beta_{k+1}$, meet the intersection point $u_k$ of 
$\beta_k$ and $\beta_{k+1}$, and finally, traverse an arc in $\beta_k$ until we 
arrive back at the initial component of $\gen{x}$. Notice that all angles here 
are at most straight angles. The space of empty pentagons $p \in \Pent_k 
(\gen{x}, \gen{y})$ with $\gen{x} \cap \Int (p) = \eset$ is denoted $\ePent_k 
(\gen{x}, \gen{y})$.

Similarly, we let $\Tri_k (\gen{x}, \gen{y})$ denote the space of embedded 
triangles with the following properties. $\Tri_k (\gen{x}, \gen{y})$ is empty 
unless $\gen{x}, \gen{y}$ coincide at exactly $n - 1$ points. An element $p$ of 
$\Tri_k (\gen{x}, \gen{y})$ is an embedded disk in $\torus$, whose boundary 
consists of three arcs, each contained in horizontal or vertical circles; under 
the orientation induced on the boundary of $p$, we start at the 
$\beta_k$-component of $\gen{x}$, traverse the arc of a horizontal circle, meet 
the component of $\gen{y}$ contained in $\beta_{k+1}$, proceed to an arc in 
$\beta_{k+1}$, meet the intersection point $v_k$ of $\beta_k$ and 
$\beta_{k+1}$, and finally traverse an arc in $\beta_k$ to return to the 
initial component of $\gen{x}$. Again, all the angles here are less than 
straight angles. It is clear that all triangles $p \in \Tri_k (\gen{x}, 
\gen{y})$ satisfy $\gen{x} \cap \Int (p) = \eset$. Furthermore, for any 
generator $\gen{x}$, there is at most one generator $\gen{y}$ such that $\Tri_k 
(\gen{x}, \gen{y})$ is not empty.

\begin{figure}
  \includegraphics[width=\textwidth]{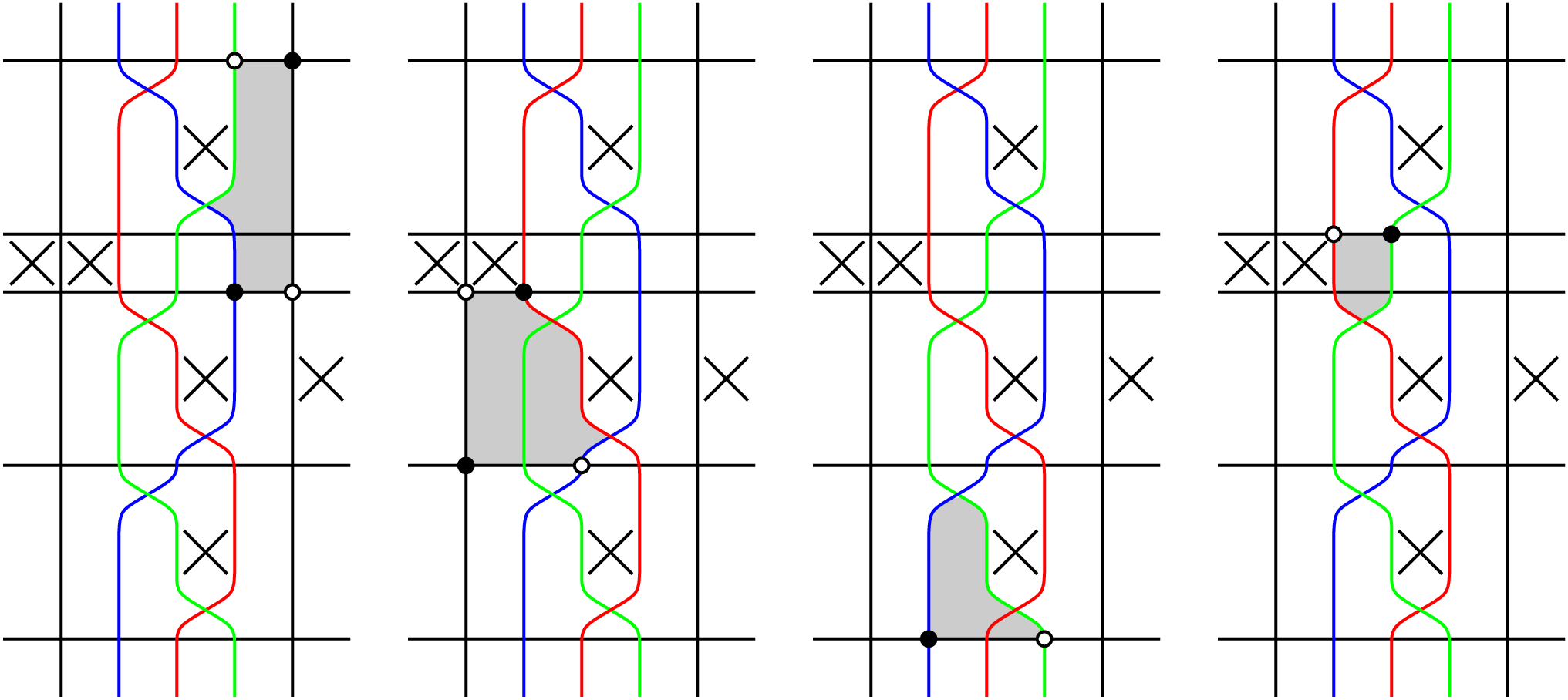}
  \caption{Two allowed pentagons in $\ePent_k (\gen{x}, \gen{y})$ and two 
    allowed triangles in $\Tri_k (\gen{x}, \gen{y})$. The components of 
    $\gen{x}$ are indicated by solid points, and those of $\gen{y}$ are 
    indicated by hollow ones.}
  \label{fig:pent_tri}
\end{figure}

Given $\gen{x} \in \gen{S} (\grid_k)$, we define
\begin{align*}
  \poly{P}_k (\gen{x}) = \sum_{\gen{y} \in \gen{S} (\grid_{k+1})} 
  \sum_{\substack{p \in \ePent_k (\gen{x}, \gen{y})\\ \Int (p) \cap \markers = 
      \eset}} \gen{y} \in C_{k+1},\\
  \poly{T}_k (\gen{x}) = \sum_{\gen{y} \in \gen{S} (\grid_{k+1})} 
  \sum_{\substack{p \in \Tri_k (\gen{x}, \gen{y})\\ \Int (p) \cap \markers = 
      \eset}} \gen{y} \in C_{k+1},
\end{align*}
and
\[
  f_k (\gen{x}) = \poly{P}_k (\gen{x}) + \poly{T}_k (\gen{x}) \in C_{k+1}.
\]

\begin{lem}
  \label{lem:chain}
  The map $f_k$ is a chain map. In fact, $\poly{P}_k$ and $\poly{T}_k$ are both 
  chain maps.
\end{lem}

\begin{proof}
  The proof is similar to that of Lemma~3.1 of \cite{MOST07}. We consider 
  domains which are obtained as the juxtaposition of a pentagon or a triangle, 
  and a rectangle. There are a few possibilities; in particular, the polygons 
  may be disjoint, their interiors may overlap, or they may share a common 
  corner. If the polygons are disjoint or if their interiors overlap, the 
  domain can be decomposed as either $r * p$ or $p *r$, and so does not 
  contribute to $\bound_{k+1} \comp f_k + f_k \comp \bound_k$. If the polygons 
  share a common corner, the resulting domain always has an alternate 
  decomposition, as shown in Figure~\ref{fig:chain}. In all cases, the domain 
  can be decomposed in two ways, and makes no contribution to $\bound_{k+1} 
  \comp f_k + f_k \comp \bound_k$.  Note that each domain has exactly two 
  decompositions, both of which are counted in $\bound_{k+1} \comp f_k + f_k 
  \comp \bound_k$; this is not going to be the case in similar lemmas later.
\end{proof}

\begin{figure}
  \includegraphics[width=0.8\textwidth]{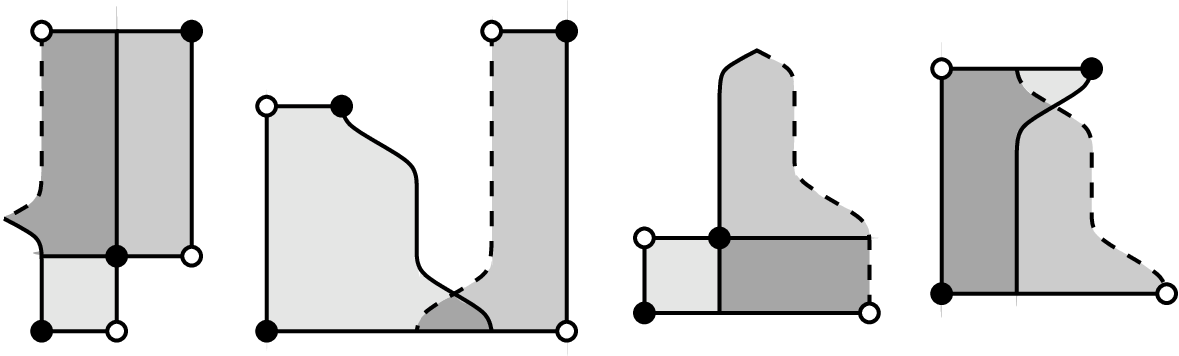}
  \caption{Two typical domains that arise as the juxtaposition of a pentagon 
    and a rectangle, and two that arise as that of a triangle and a rectangle. 
    The $\beta_k$ curve is solid while the $\beta_{k+1}$ curve is dotted.}
  \label{fig:chain}
\end{figure}

\begin{rmk}
  \label{rmk:tri_decomp}
  In the proof above, every domain that arises as the juxtaposition of a 
  triangle and a rectangle has exactly two decompositions, one contributing to 
  $\bound_{k+1} \comp \poly{T}_k$ and one to $\poly{T}_k \comp \bound_k$. The 
  analogous statement is not true for $\poly{P}_k$.
\end{rmk}

Henceforth, when considering the composition of two maps that count polygons, 
we shall ignore the cases where the two polygons are disjoint or have 
overlapping interiors, since we can always decompose the domain as either $p_1 
* p_2$ or $p_2 * p_1$ in these cases.

Next, we define the chain homotopies $\phi_k \colon C_k \to C_{k+2}$ by 
counting hexagons and quadrilaterals.

Given $\gen{x} \in \gen{S} (\grid_k)$ and $\gen{y} \in \gen{S} (\grid_{k+2})$, 
let $\Hex_k (\gen{x}, \gen{y})$ denote the space of embedded hexagons with the 
following properties. First of all, $\Hex_k (\gen{x}, \gen{y})$ is empty unless 
$\gen{x}, \gen{y}$ coincide at exactly $n - 2$ points. An element $p$ of 
$\Hex_k (\gen{x}, \gen{y})$ is an embedded disk in $\poly{T}$, whose boundary 
consists of six arcs, each contained in horizontal or vertical circles; under 
the orientation induced on the boundary of $p$, we start at the 
$\beta_k$-component of $\gen{x}$, traverse the arc of a horizontal circle, meet 
its corresponding component of $\gen{y}$, proceed to an arc of a vertical 
circle, meet the corresponding component of $\gen{x}$, continue through another 
horizontal circle, meet the component of $\gen{y}$ contained in $\beta_{k+2}$, 
proceed to an arc in $\beta_{k+2}$, meet the intersection point $u_{k+1}$ of 
$\beta_{k+1}$ and $\beta_{k+2}$, traverse an arc in $\beta_{k+1}$, meet the 
intersection point $u_k$ of $\beta_k$ and $\beta_{k+1}$, and finally, traverse 
an arc in $\beta_k$ until we arrive back at the initial component of $\gen{x}$. 
All the angles here are at most straight angles. The space of empty hexagons $p 
\in \Hex_k (\gen{x}, \gen{y})$ with $\gen{x} \cap \Int (p) = \eset$ is denoted 
$\eHex_k (\gen{x}, \gen{y})$.

Similarly, we let $\Quad_k (\gen{x}, \gen{y})$ denote the space of embedded 
quadrilaterals with the following properties. $\Quad_k (\gen{x}, \gen{y})$ is 
empty unless $\gen{x}, \gen{y}$ coincide at exactly $n - 1$ points. An element 
$p$ of $\Quad_k (\gen{x}, \gen{y})$ is an embedded disk in $\poly{T}$, whose 
boundary consists of four arcs, each contained in horizontal or vertical 
circles; under the orientation induced on the boundary of $p$, we start at the 
$\beta_k$-component of $\gen{x}$, traverse the arc of a horizontal circle, meet 
the component of $\gen{y}$ contained in $\beta_{k+2}$, proceed to an arc in 
$\beta_{k+2}$, meet the intersection point $u_{k+1}$ of $\beta_{k+1}$ and 
$\beta_{k+2}$, proceed to an arc in $\beta_{k+1}$, meet the intersection point 
$v_k$ of $\beta_k$ with $\beta_{k+1}$, and finally traverse an arc in $\beta_k$ 
to return to the initial component of $\gen{x}$. All the angles here are at 
most straight angles. It is clear that all quadrilaterals $p \in \Tri_k 
(\gen{x}, \gen{y})$ satisfy $\gen{x} \cap \Int (p) = \eset$.

\begin{figure}
  \includegraphics[width=0.741\textwidth]{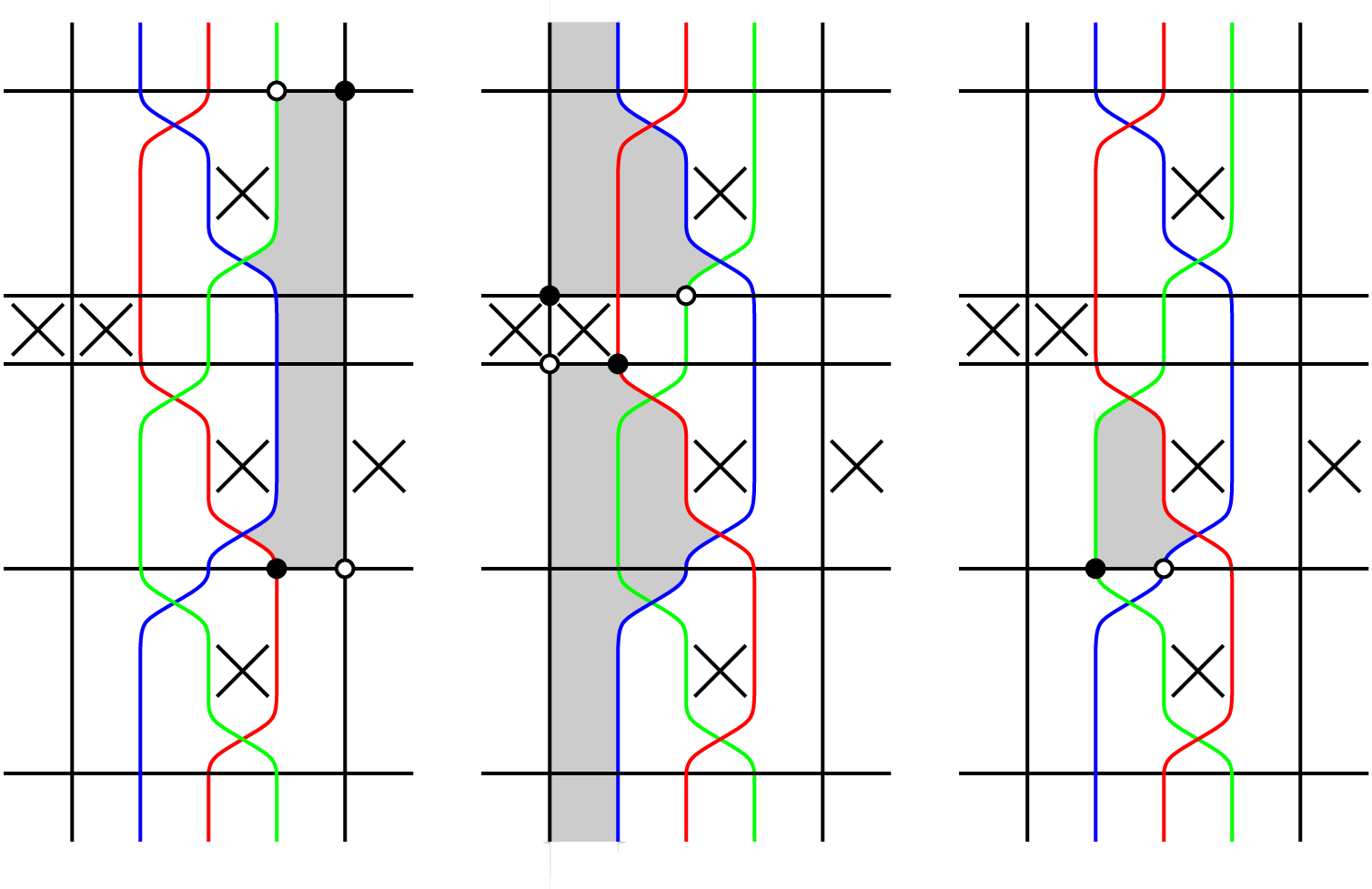}
  \caption{Two allowed hexagons in $\eHex_k (\gen{x}, \gen{y})$ and an allowed 
    quadrilateral in $\Quad_k (\gen{x}, \gen{y})$. The hexagon in the middle 
    figure is the only allowed empty hexagon that is a left domain.}
  \label{fig:hex_quad}
\end{figure}

Given $\gen{x} \in \gen{S} (\grid_k)$, we define
\begin{align*}
  \poly{H}_k (\gen{x}) = \sum_{\gen{y} \in \gen{S} (\grid_{k+2})} 
  \sum_{\substack{p \in \eHex_k (\gen{x}, \gen{y})\\ \Int (p) \cap \markers = 
      \eset}} \gen{y} \in C_{k+2},\\
  \poly{Q}_k (\gen{x}) = \sum_{\gen{y} \in \gen{S} (\grid_{k+2})} 
  \sum_{\substack{p \in \Quad_k (\gen{x}, \gen{y})\\ \Int (p) \cap \markers = 
      \eset}} \gen{y} \in C_{k+2},
\end{align*}
and
\[
  \phi_k (\gen{x}) = \poly{H}_k (\gen{x}) + \poly{Q}_k (\gen{x}) \in C_{k+2}.
\]

We say that a domain $p$ is a \emph{left domain} if $\Int (p) \cap \Int (b) = 
\eset$, and a \emph{right domain} if $\Int (p) \cap \Int (b) \neq \eset$.

\begin{lem}
  \label{lem:cond_1}
  The maps $f_k$ and $\phi_k$ satisfy Condition~(1) of Lemma~\ref{lem:hom_alg}.
\end{lem}

\begin{proof}
  Juxtaposing a triangle and a pentagon appearing in $\poly{P}_{k+1} \comp 
  \poly{T}_k$, we generically obtain a composite domain that admits a unique 
  alternative decomposition as a quadrilateral and a rectangle, appearing in 
  either $\bound_{k+2} \comp \poly{Q}_k$ or $\poly{Q}_k \comp \bound_k$, except 
  for one special case, which is described in (1) below.  Juxtaposing two 
  pentagons appearing in $\poly{P}_{k+1} \comp \poly{P}_k$, we generically 
  obtain a composite domain that admits a unique alternative decomposition as a 
  hexagon and a rectangle, appearing in either $\bound_{k+2} \comp \poly{H}_k$ 
  or $\poly{H}_k \comp \bound_k$, except for one special case, which is 
  described in (3) below.

  There are three special cases: We consider domains $p$ arising from the 
  juxtaposition of 

  \begin{figure}
    \resizebox{\textwidth}{!}{\input{fig_cond_1.pstex_t}}
    \caption{Three special cases. In each case, there are two domains $p$ and 
      $p'$, each counted exactly once in $f_{k+1} \comp f_k + \bound_{k+2} 
      \comp \phi_k + \phi_k \comp \bound_k$.}
    \label{fig:cond_1}
  \end{figure}

  \begin{enumerate}
    \item a pentagon and a triangle appearing in $\poly{T}_{k+1} \comp 
      \poly{P}_k$, such that the $\beta_k$-com\-po\-nent of $\gen{x}$ does not 
      lie on the boundary of any annular component of $\torus$ minus the 
      vertical circles (including $\beta_\infty, \beta_0, \beta_1$). Visually, 
      the $\beta_k$-component of $\gen{x}$ lies on the central vertical axis of 
      the figure. The domain $p$ can only be alternatively decomposed as the 
      triangle $t_{k+2}$ and a pentagon, bounded by $\beta_k$, a horizontal 
      arc, a vertical arc, another horizontal arc and $\beta_{k+2}$, in its 
      induced orientation. The pentagon is in the opposite orientation as one 
      that would be counted in the map $f_{k+2}$, and its boundary meets only 
      the intersection point $v_{k+2}$; such a pentagon is not counted in any 
      map.  The triangle $t_{k+2}$ is not counted in any map either. Therefore, 
      $p$ is counted exactly once in $f_{k+1} \comp f_k + \bound_{k+2} \comp 
      \phi_k + \phi_k \comp \bound_k$.  However, we can replace $t_{k+2}$ with 
      the triangle $t_k$, and obtain a corresponding domain $p'$ that also 
      connects $\gen{x}$ to $\gen{y}$.  The new domain $p'$ admits a unique 
      alternative decomposition as a triangle and a pentagon, counted in 
      $\poly{P}_{k+1} \comp \poly{T}_k$.  See Figure~\ref{fig:cond_1}~(1).

    \item a pentagon and a triangle appearing in $\poly{T}_{k+1} \comp 
      \poly{P}_k$, such that the $\beta_k$-com\-po\-nent of $\gen{x}$ lies on 
      the boundary of an annular component of $\torus$ minus the vertical 
      circles (including $\beta_\infty, \beta_0, \beta_1$).  Visually, the 
      $\beta_k$-component of $\gen{x}$ lies to the left of the central vertical 
      axis of the figure. The domain $p$ can only be alternatively decomposed 
      as the triangle $t_{k+2}$ and a pentagon, bounded by $\beta_k$, a 
      horizontal arc, a vertical arc, another horizontal arc and $\beta_{k+2}$, 
      in its induced orientation.  The pentagon is in the opposite orientation 
      as one that would be counted in the map $f_{k+2}$, and its boundary meets 
      only the intersection point $v_{k+2}$; such a pentagon is not counted in 
      any map.  The triangle $t_{k+2}$ is not counted in any map either.  
      Therefore, $p$ is counted exactly once in $f_{k+1} \comp f_k + 
      \bound_{k+2} \comp \phi_k + \phi_k \comp \bound_k$.  However, we can 
      replace $t_{k+2}$ with the triangle $t_k$, and obtain a corresponding 
      domain $p'$ that also connects $\gen{x}$ to $\gen{y}$.  The new domain 
      $p'$ admits a unique alternative decomposition as a quadrilateral and a 
      rectangle, counted in $\bound_{k+2} \comp \poly{Q}_k$.  See 
      Figure~\ref{fig:cond_1}~(2).

    \item two triangles appearing in $\poly{T}_{k+1} \comp \poly{T}_k$. The 
      domain $p$ can only be alternatively decomposed as the triangle $t_{k+1}$ 
      and another triangle, bounded by segments of $\beta_k, \beta_{k+2}$ and a 
      horizontal circle in its induced orientation, and having $u_{k+2}$ as a 
      corner.  Since neither triangle is counted in any of the maps $f_k$ or 
      $\phi_k$, $p$ is counted exactly once in $f_{k+1} \comp f_k + 
      \bound_{k+2} \comp \phi_k + \phi_k \comp \bound_k$.  However, we can 
      replace $t_{k+1}$ with the annulus $b$, and obtain a corresponding domain 
      $p'$ that also connects $\gen{x}$ to $\gen{y}$.  The situation is similar 
      to the special case in the proof of Lemma~3.1 of \cite{MOST07}. Depending 
      on the initial point $\gen{x}$, the new domain $p'$ admits a unique 
      alternative decomposition as two pentagons or as a hexagon and a 
      rectangle, counted either in $\poly{P}_{k+1} \comp \poly{P}_k$, in 
      $\bound_{k+2} \comp \poly{H}_k$, or in $\poly{H}_k \comp \bound_k$. See 
      Figure~\ref{fig:cond_1}~(3).
  \end{enumerate}

  \begin{table}
    \captionsetup{belowskip=10pt}
    \centering
    \begin{tabular}{|c|c|c||c|c|}
      \hline
      \multicolumn{2}{|c|}{Term in} & Position & Cancels with Term in & Sp.\ 
      Case\\
      \hline \hline
      \multirow{9}{*}{$f_{k+1} \comp f_k$} & \multirow{2}{*}{$\poly{P}_{k+1} 
        \comp \poly{P}_k$} & left, right ($\not\supset b$) & $\bound_{k+2} 
      \comp \poly{H}_k$ or $\poly{H}_k \comp \bound_k$ &\\
      \cline{3-5}
      & & right ($\supset b$) & $\poly{T}_{k+1} \comp \poly{T}_k$ & (3)\\
      \cline{2-5}
      & \multirow{2}{*}{$\poly{T}_{k+1} \comp \poly{P}_k$} & left (central $x$) 
      & $\poly{P}_{k+1} \comp \poly{T}_k$ & (1)\\
      \cline{3-5}
      & & left (left $x$) & $\bound_{k+2} \comp \poly{Q}_k$ & (2)\\
      \cline{2-5}
      & \multirow{3}{*}{$\poly{P}_{k+1} \comp \poly{T}_k$} & left (central $y$) 
      & $\poly{Q}_k \comp \bound_k$ &\\
      \cline{3-5}
      & & left (otherwise) & $\poly{T}_{k+1} \comp \poly{P}_k$ & (1)\\
      \cline{3-5}
      & & right & $\bound_{k+2} \comp \poly{Q}_k$ &\\
      \cline{2-5}
      & \multirow{2}{*}{$\poly{T}_{k+1} \comp \poly{T}_k$} & 
      \multirow{2}{*}{left} & $\poly{P}_{k+1} \comp \poly{P}_k$, $\bound_{k+2} 
      \comp \poly{H}_k$ & \multirow{2}{*}{(3)}\\
      & & & or $\poly{H}_k \comp \bound_k$ &\\
      \hline
      \multirow{6}{*}{$\bound_{k+2} \comp \phi_k$} & 
      \multirow{3}{*}{$\bound_{k+2} \comp \poly{H}_k$} & left & $\poly{P}_{k+1} 
      \comp \poly{P}_k$ &\\
      \cline{3-5}
      & & right ($\not\supset b$) & $\poly{P}_{k+1} \comp \poly{P}_k$ or 
      $\poly{H}_k \comp \bound_k$ &\\
      \cline{3-5}
      & & right ($\supset b$) & $\poly{T}_{k+1} \comp \poly{T}_k$ & (3)\\
      \cline{2-5}
      & \multirow{3}{*}{$\bound_{k+2} \comp \poly{Q}_k$} & left (central $y$) & 
      $\poly{Q}_k \comp \bound_k$ &\\
      \cline{3-5}
      & & left (otherwise) & $\poly{T}_{k+1} \comp \poly{P}_k$ & (2)\\
      \cline{3-5}
      & & right & $\poly{P}_{k+1} \comp \poly{T}_k$ or $\poly{Q}_k \comp 
      \bound_k$ &\\
      \hline
      \multirow{5}{*}{$\phi_k \comp \bound_k$} & \multirow{3}{*}{$\poly{H}_k 
        \comp \bound_k$} & left & $\poly{P}_{k+1} \comp \poly{P}_k$ &\\
      \cline{3-5}
      & & right ($\not\supset b$) & $\poly{P}_{k+1} \comp \poly{P}_k$ or 
      $\bound_{k+2} \comp \poly{H}_k$ &\\
      \cline{3-5}
      & & right ($\supset b$) & $\poly{T}_{k+1} \comp \poly{T}_k$ & (3)\\
      \cline{2-5}
      & \multirow{2}{*}{$\poly{Q}_k \comp \bound_k$} & left & $\poly{P}_{k+1} 
      \comp \poly{T}_k$ or $\bound_{k+2} \comp \poly{Q}_k$ &\\
      \cline{3-5}
      & & right & $\bound_{k+2} \comp \poly{Q}_k$ &\\
      \hline
    \end{tabular}
    \caption{This table shows how the terms cancel each other in 
      Lemma~\ref{lem:cond_1}. A left domain is indicated as ``(central $x$)'' 
      if the $\beta_k$-component of $\gen{x}$ lies on the central vertical axis 
      of the figure, and ``(left $x$)'' if it lies to the left of this axis; 
      similarly for $\gen{y}$. A right domain is indicated as ``($\supset b$)'' 
      if $\Int (p) \supset \Int (b)$, and ``($\not\supset b$)'' otherwise.  The 
      special cases are shown in Figure~\ref{fig:cond_1}.}
    \label{tab:cond_1}
  \end{table}

  Finally, the remaining terms in $\bound_{k+2} \comp \poly{Q}_k$ cancel with 
  terms in $\poly{Q}_k \comp \bound_k$, and the remaining terms in 
  $\bound_{k+2} \comp \poly{H}_k$ cancel with terms in $\poly{H}_k \comp 
  \bound_k$.  Table~\ref{tab:cond_1} summarizes how the terms cancel each 
  other.
\end{proof}

We now define the chain homotopy $\psi_k \colon C_k \to C_k$ by counting 
heptagons.

Given $\gen{x} \in \gen{S} (\grid_k)$ and $\gen{y} \in \gen{S} (\grid_{k+2})$, 
let $\Hept_k (\gen{x}, \gen{y})$ denote the space of embedded heptagons. First 
of all, $\Hept_k (\gen{x}, \gen{y})$ is empty unless $\gen{x}, \gen{y}$ 
coincide at exactly $n - 2$ points. An element $p$ of $\Hept_k (\gen{x}, 
\gen{y})$ is an embedded disk in $\torus$, whose boundary consists of seven 
arcs, each contained in horizontal or vertical circles; under the orientation 
induced on the boundary of $p$, we start at the $\beta_k$-component of 
$\gen{x}$, traverse the arc of a horizontal circle, meet its corresponding 
component of $\gen{y}$, proceed to an arc of a vertical circle, meet the 
corresponding component of $\gen{x}$, continue through another horizontal 
circle, meet the component of $\gen{y}$ contained in $\beta_k$, proceed to an 
arc in $\beta_k$, meet the intersection point $u_{k+2}$ of $\beta_k$ and 
$\beta_{k+2}$, proceed to an arc in $\beta_{k+2}$, meet the intersection point 
$u_{k+1}$ of $\beta_{k+1}$ and $\beta_{k+2}$, traverse an arc in $\beta_{k+1}$, 
meet the intersection point $u_k$ of $\beta_k$ and $\beta_{k+1}$, and finally, 
traverse an arc in $\beta_k$ until we arrive back at the initial component of 
$\gen{x}$. All the angles here are again at most straight angles. The space of 
empty heptagons $p \in \Hept_k (\gen{x}, \gen{y})$ with $\gen{x} \cap \Int (p) 
= \eset$ is denoted $\eHept_k (\gen{x}, \gen{y})$.

\begin{figure}
  \includegraphics[width=0.222\textwidth]{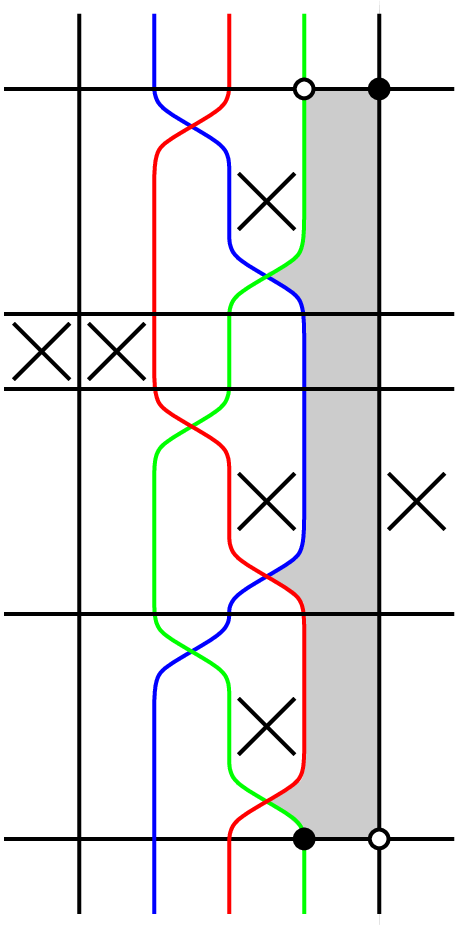}
  \caption{An allowed heptagon in $\eHept_k (\gen{x}, \gen{y})$. Note that 
    allowed heptagons are necessarily right domains.}
  \label{fig:hept}
\end{figure}

Given $\gen{x} \in \gen{S} (\grid_k)$, we define
\[
  \psi_k (\gen{x}) = \poly{K}_k (\gen{x}) = \sum_{\gen{y} \in \gen{S} 
    (\grid_k)} \sum_{\substack{p \in \eHept_k (\gen{x}, \gen{y})\\ \Int (p) 
      \cap \markers = \eset}} \gen{y} \in C_k.
\]

\begin{lem}
  \label{lem:cond_2}
  We have
  \[
    \phi_{k+1} \comp f_k + f_{k+2} \comp \phi_k + \bound_k \comp \psi_k + 
    \psi_k \comp \bound_k = \id,
  \]
  so that the maps $f_k$ and $\phi_k$ satisfy Condition~(2) of 
  Lemma~\ref{lem:hom_alg}.
\end{lem}

\begin{proof}
  First, we see that juxtaposing either a triangle and a hexagon, or a pentagon 
  and a quadrilateral, does not contribute to $\phi_{k+1} \comp f_k + f_{k+2} 
  \comp \phi_k$. In other words, we first claim that
  \[
    \poly{P}_{k+2} \comp \poly{Q}_k + \poly{Q}_{k+1} \comp \poly{P}_k + 
    \poly{H}_{k+1} \comp \poly{T}_k + \poly{T}_{k+2} \comp \poly{H}_k = 0.
  \]
  There are exactly four cases:  We consider domains $p$ formed by juxtaposing

  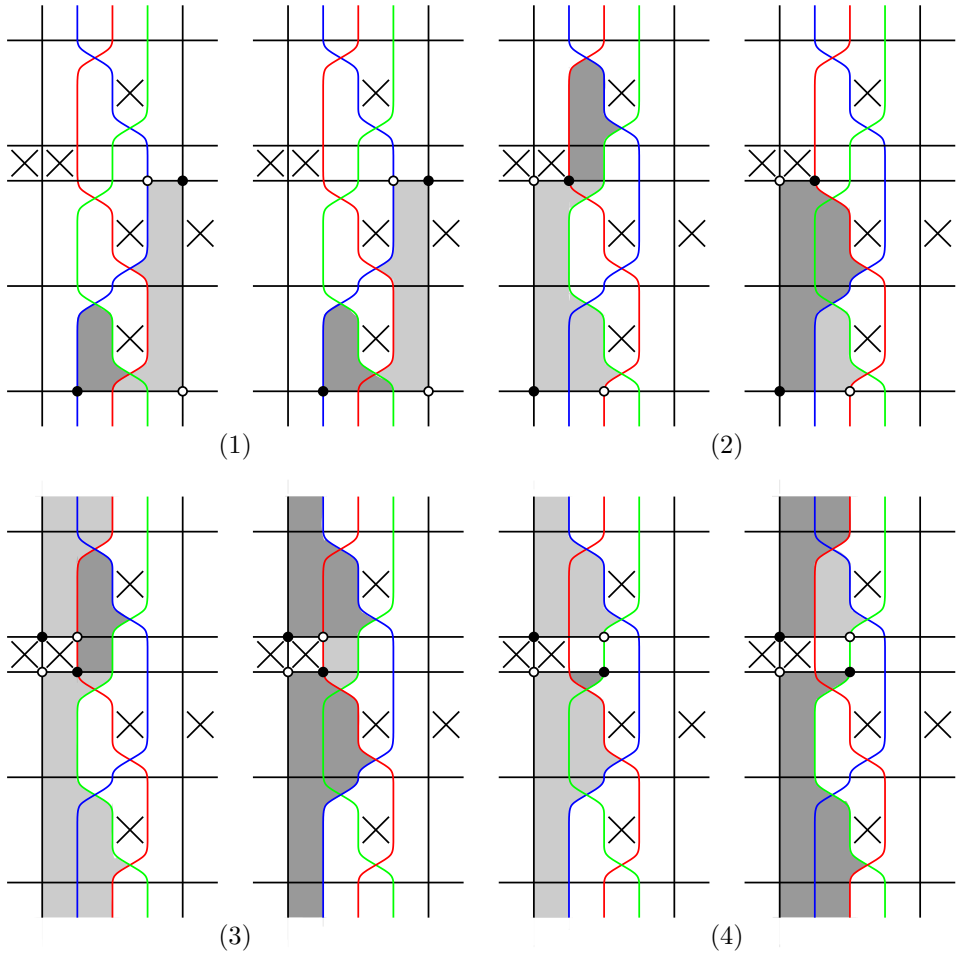
\begin{figure}
    \resizebox{\textwidth}{!}{\input{fig_cond_2.pstex_t}}
    \caption{The four cases when juxtaposing a triangle and a hexagon, or a 
      pentagon and a quadrilateral. All terms that arise cancel out with each 
      other.}
    \label{fig:cond_2}
  \end{figure}

  \begin{enumerate}
    \item a quadrilateral and a pentagon appearing in $\poly{P}_{k+2} \comp 
      \poly{Q}_k$, such that $p$ is a right domain. In this case, $p$ admits a 
      unique alternative decomposition as a triangle and a hexagon appearing in 
      $\poly{H}_{k+1} \comp \poly{T}_k$. See Figure~\ref{fig:cond_2}~(1);

    \item a quadrilateral and a pentagon appearing in $\poly{P}_{k+2} \comp 
      \poly{Q}_k$, such that $p$ is a left domain with height less than $n$. 
      Only this decomposition of $p$ is counted. However, we can replace the 
      triangle $t_k$ inside $p$ by the triangle $t_{k+2}$, and obtain a 
      corresponding domain $p'$ that can be uniquely decomposed as a pentagon 
      and a quadrilateral appearing in $\poly{Q}_{k+1} \comp \poly{P}_k$. See 
      Figure~\ref{fig:cond_2}~(2);

    \item a quadrilateral and a pentagon appearing in $\poly{P}_{k+2} \comp 
      \poly{Q}_k$, such that $p$ is a right domain with height $n$. This is in 
      fact only possible when $k = 1$. Only this decomposition of $p$ is 
      counted. However, $p$ contains the triangles $t_k$ and $t_{k+1}$; we can 
      replace $t_{k+1}$ by the triangle $t_{k+2}$, and obtain a corresponding 
      domain $p'$ that can be uniquely decomposed as a hexagon and a triangle 
      appearing in $\poly{T}_{k+2} \comp \poly{H}_k$. See 
      Figure~\ref{fig:cond_2}~(3).

      The astute reader may find it strange that in this particular case, in 
      $p$ the triangle $t_k$ is ``attached'' to the top of the rest of the 
      domain, whereas in $p'$ it is ``attached'' to the bottom. One way to 
      convince oneself of the validity of the procedure of obtaining $p'$ from 
      $p$, is to think of it as first replacing $t_k$ by $t_{k+2}$, and then 
      replacing $t_{k+1}$ by $t_k$;

    \item a triangle and a hexagon appearing in $\poly{H}_{k+1} \comp 
      \poly{T}_k$, such that $p$ is a left domain. This is only possible when 
      $k = 0$. Only this decomposition of $p$ is counted. However, we can 
      replace the triangle $t_k$ inside $p$ by the triangle $t_{k+2}$, and 
      obtain a corresponding domain $p'$ that can be uniquely decomposed as a 
      pentagon and a quadrilateral appearing in $\poly{Q}_{k+1} \comp 
      \poly{P}_k$. See Figure~\ref{fig:cond_2}~(4).
  \end{enumerate}

  Juxtaposing a pentagon and a hexagon appearing in $\poly{H}_{k+1} \comp 
  \poly{P}_k$ or $\poly{P}_{k+2} \comp \poly{H}_k$, we generically obtain a 
  composite domain that admits a unique alternative decomposition as a heptagon 
  and a rectangle, appearing in $\bound_k \comp \poly{K}_k$ or $\poly{K}_k 
  \comp \bound_k$, except for one special case discussed below.

  \begin{figure}
    \includegraphics[width=\textwidth]{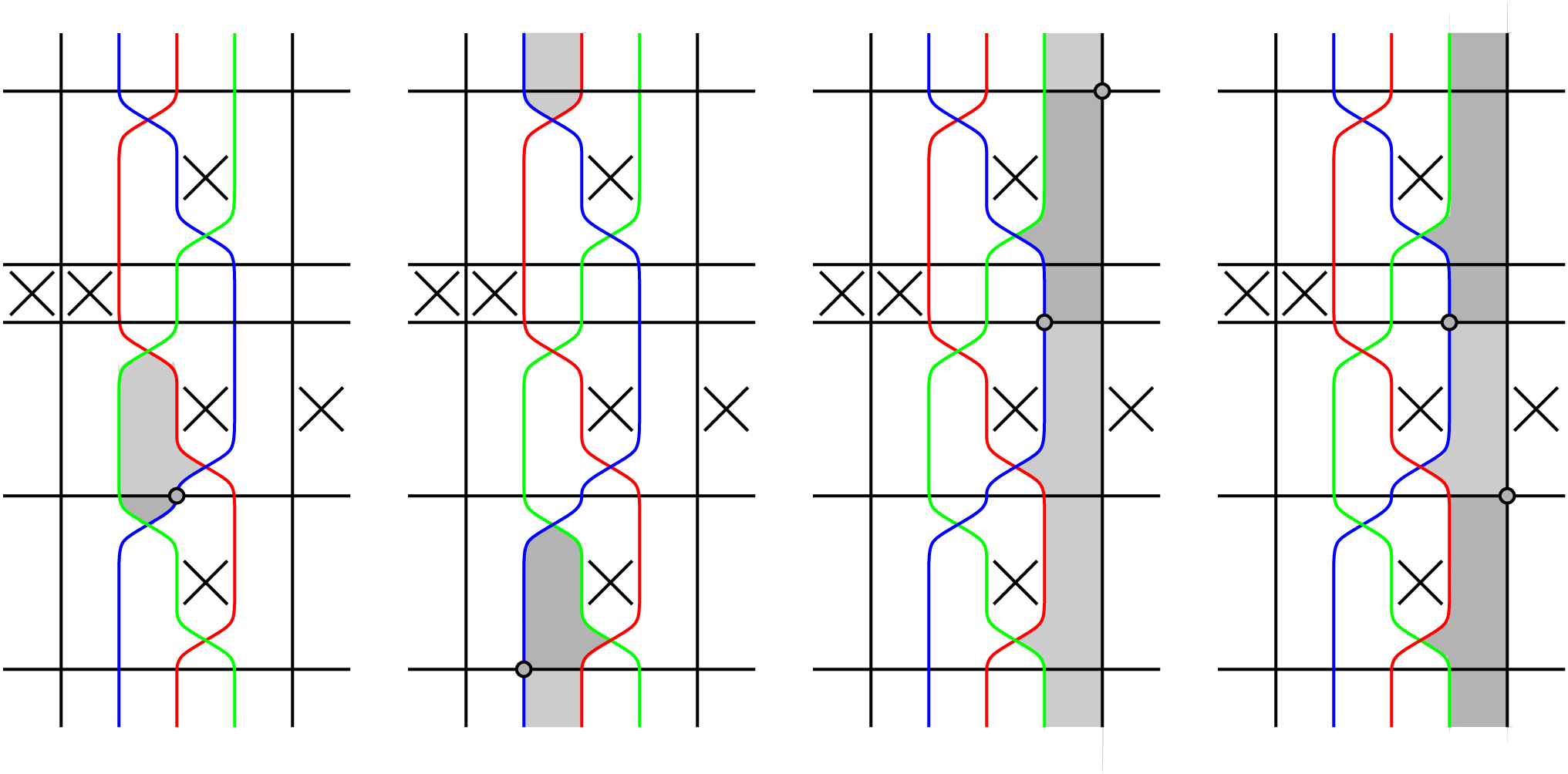}
    \caption{Decomposing the identity map.}
    \label{fig:iden}
  \end{figure}

  Depending on the initial point $\gen{x}$, there exists exactly one domain $p$ 
  connecting $\gen{x}$ to itself that admits a unique decomposition, either as 
  a triangle and a quadrilateral in $\poly{Q}_{k+1} \comp \poly{T}_k$ (in which 
  case $p$ is the triangle $t_{k+1}$), as a quadrilateral and a triangle in 
  $\poly{T}_{k+2} \comp \poly{Q}_k$ ($p$ is the triangle $t_k$), as a pentagon 
  and a hexagon appearing in $\poly{H}_{k+1} \comp \poly{P}_k$ or 
  $\poly{P}_{k+2} \comp \poly{H}_k$ ($p$ is the annulus $b$), or as a rectangle 
  and a heptagon appearing in $\poly{K}_k \comp \bound_k$ or $\bound_k \comp 
  \poly{K}_k$ ($p$ is again the annulus $b$).  Of course, in this case, the 
  identity map is counted once.  See Figure~\ref{fig:iden}.

  \begin{table}
    \captionsetup{belowskip=10pt}
    \centering
    \begin{tabular}{|c|c|c||c|c|}
      \hline
      \multicolumn{2}{|c|}{Term in} & Position & Cancels with Term in & Sp.\ 
      Case\\
      \hline \hline
      \multirow{7}{*}{$\phi_{k+1} \comp f_k$} & \multirow{2}{*}{$\poly{H}_{k+1} 
        \comp \poly{P}_k$} & right ($\neq b$) & $\bound_k \comp \poly{K}_k$ or 
      $\poly{K}_k \comp \bound_k$ &\\
      \cline{3-5}
      & & right ($= b$) & $\id$ &\\
      \cline{2-5}
      & \multirow{2}{*}{$\poly{Q}_{k+1} \comp \poly{P}_k$} & left (ht.\ $< n - 
      1$) & $\poly{P}_{k+2} \comp \poly{Q}_k$ & (2)\\
      \cline{3-5}
      & & left (ht.\ $= n - 1$) & $\poly{H}_{k+1} \comp \poly{T}_k$ & (4)\\
      \cline{2-5}
      & \multirow{2}{*}{$\poly{H}_{k+1} \comp \poly{T}_k$} & left & 
      $\poly{Q}_{k+1} \comp \poly{P}_k$ & (4)\\
      \cline{3-5}
      & & right & $\poly{P}_{k+2} \comp \poly{Q}_k$ & (1)\\
      \cline{2-5}
      & $\poly{Q}_{k+1} \comp \poly{T}_k$ & left & $\id$ &\\
      \hline
      \multirow{7}{*}{$f_{k+2} \comp \phi_k$} & \multirow{2}{*}{$\poly{P}_{k+2} 
        \comp \poly{H}_k$} & right ($\neq b$) & $\bound_k \comp \poly{K}_k$ or 
      $\poly{K}_k \comp \bound_k$ &\\
      \cline{3-5}
      & & right ($= b$) & $\id$ &\\
      \cline{2-5}
      & $\poly{T}_{k+2} \comp \poly{H}_k$ & left & $\poly{P}_{k+2} \comp 
      \poly{Q}_k$ & (3)\\
      \cline{2-5}
      & \multirow{3}{*}{$\poly{P}_{k+2} \comp \poly{Q}_k$} & left (ht.\ $< n$) 
      & $\poly{Q}_{k+1} \comp \poly{P}_k$ & (2)\\
      \cline{3-5}
      & & left (ht.\ $= n$) & $\poly{T}_{k+1} \comp \poly{H}_k$ & (3)\\
      \cline{3-5}
      & & right & $\poly{H}_{k+1} \comp \poly{T}_k$ & (1)\\
      \cline{2-5}
      & $\poly{T}_{k+2} \comp \poly{Q}_k$ & left & $\id$ &\\
      \hline
      \multirow{2}{*}{$\bound_k \comp \psi_k$} & \multirow{2}{*}{$\bound_k 
        \comp \poly{K}_k$} & \multirow{2}{*}{right} & $\poly{H}_{k+1} \comp 
      \poly{P}_k$, $\poly{P}_{k+2} \comp \poly{H}_k$, &\\
      & & & $\poly{K}_k \comp \bound_k$ or $\id$ &\\
      \hline
      \multirow{2}{*}{$\psi_k \comp \bound_k$} & \multirow{2}{*}{$\poly{K}_k 
        \comp \bound_k$} & \multirow{2}{*}{right} & $\poly{H}_{k+1} \comp 
      \poly{P}_k$, $\poly{P}_{k+2} \comp \poly{H}_k$, &\\
      & & & $\bound_k \comp \poly{K}_k$ or $\id$ &\\
      \hline
      \multicolumn{2}{|c|}{\multirow{4}{*}{$\id$}} & \multirow{4}{*}{N/A} & 
      $\poly{H}_{k+1} \comp \poly{P}_k$, $\poly{P}_{k+2} \comp \poly{H}_k$, &\\
      \multicolumn{2}{|c|}{} & & $\bound_k \comp \poly{K}_k$, $\poly{K}_k \comp 
      \bound_k$, &\\
      \multicolumn{2}{|c|}{} & & $\poly{Q}_{k+1} \comp \poly{T}_k$ &\\
      \multicolumn{2}{|c|}{} & & or $\poly{T}_{k+2} \comp \poly{Q}_k$ &\\
      \hline
    \end{tabular}
    \caption{This table shows how the terms cancel each other in 
      Lemma~\ref{lem:cond_2}. The special cases are shown in 
      Figure~\ref{fig:cond_2}.}
    \label{tab:cond_2}
  \end{table}

  Finally, the remaining terms in $\bound_k \comp \poly{K}_k$ cancel with terms 
  in $\poly{K}_k \comp \bound_k$. Table~\ref{tab:cond_2} summarizes how the 
  terms above cancel each other.
\end{proof}

The proof of Proposition~\ref{prop:exact} is completed by combining 
Lemmas~\ref{lem:hom_alg}, \ref{lem:chain}, \ref{lem:cond_1} and 
\ref{lem:cond_2}.

%% file: fig_grid_alt.pstex_t
\begin{picture}(0,0)%
\includegraphics{fig_grid_alt.pstex}%
\end{picture}%
\setlength{\unitlength}{1934sp}%
\begingroup\makeatletter\ifx\SetFigFont\undefined%
\gdef\SetFigFont#1#2#3#4#5{%
  \reset@font\fontsize{#1}{#2pt}%
  \fontfamily{#3}\fontseries{#4}\fontshape{#5}%
  \selectfont}%
\fi\endgroup%
\begin{picture}(9666,5947)(1768,-6875)
\put(10276,-3211){\makebox(0,0)[b]{\smash{{\SetFigFont{10}{12.0}{\familydefault}{\mddefault}{\updefault}{\color[rgb]{0,0,0}$L_1$}%
}}}}
\put(6601,-6811){\makebox(0,0)[b]{\smash{{\SetFigFont{10}{12.0}{\familydefault}{\mddefault}{\updefault}{\color[rgb]{0,0,0}$\grid_0$}%
}}}}
\put(10201,-6811){\makebox(0,0)[b]{\smash{{\SetFigFont{10}{12.0}{\familydefault}{\mddefault}{\updefault}{\color[rgb]{0,0,0}$\grid_1$}%
}}}}
\put(3001,-6811){\makebox(0,0)[b]{\smash{{\SetFigFont{10}{12.0}{\familydefault}{\mddefault}{\updefault}{\color[rgb]{0,0,0}$\grid_\infty$}%
}}}}
\put(2926,-3211){\makebox(0,0)[b]{\smash{{\SetFigFont{10}{12.0}{\familydefault}{\mddefault}{\updefault}{\color[rgb]{0,0,0}$L_\infty$}%
}}}}
\put(6601,-3211){\makebox(0,0)[b]{\smash{{\SetFigFont{10}{12.0}{\familydefault}{\mddefault}{\updefault}{\color[rgb]{0,0,0}$L_0$}%
}}}}
\end{picture}%

%% file: fig_one_grid.pstex_t
\begin{picture}(0,0)%
\includegraphics{fig_one_grid.pstex}%
\end{picture}%
\setlength{\unitlength}{2092sp}%
\begingroup\makeatletter\ifx\SetFigFont\undefined%
\gdef\SetFigFont#1#2#3#4#5{%
  \reset@font\fontsize{#1}{#2pt}%
  \fontfamily{#3}\fontseries{#4}\fontshape{#5}%
  \selectfont}%
\fi\endgroup%
\begin{picture}(3666,7530)(1168,-6694)
\put(3601,689){\makebox(0,0)[b]{\smash{{\SetFigFont{10}{12.0}{\familydefault}{\mddefault}{\updefault}{\color[rgb]{0,0,0}$\beta_0$}%
}}}}
\put(3901,-3211){\makebox(0,0)[b]{\smash{{\SetFigFont{10}{12.0}{\familydefault}{\mddefault}{\updefault}{\color[rgb]{0,0,0}$b$}%
}}}}
\put(2701,-5911){\makebox(0,0)[b]{\smash{{\SetFigFont{10}{12.0}{\familydefault}{\mddefault}{\updefault}{\color[rgb]{0,0,0}$t_\infty$}%
}}}}
\put(3001,689){\makebox(0,0)[b]{\smash{{\SetFigFont{10}{12.0}{\familydefault}{\mddefault}{\updefault}{\color[rgb]{0,0,0}$\beta_1$}%
}}}}
\put(2401,689){\makebox(0,0)[b]{\smash{{\SetFigFont{10}{12.0}{\familydefault}{\mddefault}{\updefault}{\color[rgb]{0,0,0}$\beta_\infty$}%
}}}}
\put(2701,-3811){\makebox(0,0)[b]{\smash{{\SetFigFont{10}{12.0}{\familydefault}{\mddefault}{\updefault}{\color[rgb]{0,0,0}$t_0$}%
}}}}
\put(2701,-1561){\makebox(0,0)[b]{\smash{{\SetFigFont{10}{12.0}{\familydefault}{\mddefault}{\updefault}{\color[rgb]{0,0,0}$t_1$}%
}}}}
\put(3676,-1636){\makebox(0,0)[b]{\smash{{\SetFigFont{10}{12.0}{\familydefault}{\mddefault}{\updefault}{\color[rgb]{0,0,0}$u_\infty$}%
}}}}
\put(3676,-4036){\makebox(0,0)[b]{\smash{{\SetFigFont{10}{12.0}{\familydefault}{\mddefault}{\updefault}{\color[rgb]{0,0,0}$u_1$}%
}}}}
\put(3676,-5836){\makebox(0,0)[b]{\smash{{\SetFigFont{10}{12.0}{\familydefault}{\mddefault}{\updefault}{\color[rgb]{0,0,0}$u_0$}%
}}}}
\put(2326,-4636){\makebox(0,0)[b]{\smash{{\SetFigFont{10}{12.0}{\familydefault}{\mddefault}{\updefault}{\color[rgb]{0,0,0}$v_\infty$}%
}}}}
\put(2326,-2836){\makebox(0,0)[b]{\smash{{\SetFigFont{10}{12.0}{\familydefault}{\mddefault}{\updefault}{\color[rgb]{0,0,0}$v_0$}%
}}}}
\put(2326,-436){\makebox(0,0)[b]{\smash{{\SetFigFont{10}{12.0}{\familydefault}{\mddefault}{\updefault}{\color[rgb]{0,0,0}$v_1$}%
}}}}
\end{picture}%

%% file: fig_cond_1.pstex_t
\begin{picture}(0,0)%
\includegraphics{fig_cond_1.pstex}%
\end{picture}%
\setlength{\unitlength}{1460sp}%
\begingroup\makeatletter\ifx\SetFigFont\undefined%
\gdef\SetFigFont#1#2#3#4#5{%
  \reset@font\fontsize{#1}{#2pt}%
  \fontfamily{#3}\fontseries{#4}\fontshape{#5}%
  \selectfont}%
\fi\endgroup%
\begin{picture}(16257,16148)(-3032,-15575)
\put(901,-7111){\makebox(0,0)[b]{\smash{{\SetFigFont{10}{12.0}{\familydefault}{\mddefault}{\updefault}{\color[rgb]{0,0,0}(1)}%
}}}}
\put(5101,-15511){\makebox(0,0)[b]{\smash{{\SetFigFont{10}{12.0}{\familydefault}{\mddefault}{\updefault}{\color[rgb]{0,0,0}(3)}%
}}}}
\put(9301,-7111){\makebox(0,0)[b]{\smash{{\SetFigFont{10}{12.0}{\familydefault}{\mddefault}{\updefault}{\color[rgb]{0,0,0}(2)}%
}}}}
\end{picture}%

%% file: fig_cond_2.pstex_t
\begin{picture}(0,0)%
\includegraphics{fig_cond_2.pstex}%
\end{picture}%
\setlength{\unitlength}{1460sp}%
\begingroup\makeatletter\ifx\SetFigFont\undefined%
\gdef\SetFigFont#1#2#3#4#5{%
  \reset@font\fontsize{#1}{#2pt}%
  \fontfamily{#3}\fontseries{#4}\fontshape{#5}%
  \selectfont}%
\fi\endgroup%
\begin{picture}(16266,16147)(-3032,-15575)
\put(901,-7111){\makebox(0,0)[b]{\smash{{\SetFigFont{10}{12.0}{\familydefault}{\mddefault}{\updefault}{\color[rgb]{0,0,0}(1)}%
}}}}
\put(9301,-7111){\makebox(0,0)[b]{\smash{{\SetFigFont{10}{12.0}{\familydefault}{\mddefault}{\updefault}{\color[rgb]{0,0,0}(2)}%
}}}}
\put(901,-15511){\makebox(0,0)[b]{\smash{{\SetFigFont{10}{12.0}{\familydefault}{\mddefault}{\updefault}{\color[rgb]{0,0,0}(3)}%
}}}}
\put(9301,-15511){\makebox(0,0)[b]{\smash{{\SetFigFont{10}{12.0}{\familydefault}{\mddefault}{\updefault}{\color[rgb]{0,0,0}(4)}%
}}}}
\end{picture}%

%% file: sec_signs.tex
\section{Signs}
\label{sec:signs}

In \cite{MOST07}, sign refinements are given to extend the definition of 
combinatorial knot Floer homology to one with coefficients in $\Z$. In this 
section, we shall likewise assign sign refinements to our maps, to prove the 
analogous statement of Proposition~\ref{prop:exact} with coefficients in $\Z$.

Given a grid diagram, we denote by $\eRect$ the union of $\eRect (\gen{x}, 
\gen{y})$ for all $\gen{x}, \gen{y}$. We follow \cite{MOST07} and adopt the 
following definition:

\begin{defn}
A \emph{true sign assignment}, or simply a \emph{sign assignment}, is a 
function
\[
\sign \colon \eRect \to \set{\pm 1}
\]
with the following properties:

\begin{enumerate}
\label{defn:sign}
\item For any four distinct $r_1, r_2, r_1', r_2' \in \eRect$ with $r_1 * r_2 = 
  r_1' * r_2'$, we have：
\[
\sign (r_1) \cdot \sign (r_2) = - \sign (r_1') \cdot \sign (r_2');
\]

\item For $r_1, r_2 \in \eRect$ such that $r_1 * r_2$ is a vertical annulus, we 
  have
\[
\sign (r_1) \cdot \sign (r_2) = - 1;
\]

\item For $r_1, r_2 \in \eRect$ such that $r_1 * r_2$ is a horizontal annulus, 
  we have
\[
\sign (r_1) \cdot \sign (r_2) = + 1.
\]
\end{enumerate}
\end{defn}

\begin{thm}[Manolescu--Ozsv\'{a}th--Szab\'{o}--Thurston]
\label{thm:sign}
There exists a sign assignment as defined in Definition~\ref{defn:sign}.  
Moreover, this sign assignment is essentially unique: If $\sign_1$ and 
$\sign_2$ are two sign assignments, then there is a function $g \colon \gen{S} 
(\grid) \to \set{\pm 1}$ so that for all $r \in \eRect (\gen{x}, \gen{y})$, 
$\sign_2 (r) = g (\gen{x}) \cdot g (\gen{y}) \cdot \sign_1 (r)$.
\end{thm}

\begin{rmk}
  \label{rmk:sign}
  In the construction of a sign assignment in \cite{MOST07}, the sign of a 
  rectangle does not depend on the positions of the $O$'s and $X$'s of the 
  diagram. We can view each generator $\gen{x}$ as a permutation 
  $\sigma_\gen{x}$; If the component of $\gen{x}$ on the $i$-th horizontal 
  circle lies on the $s (i)$-th vertical circle, then we let $\sigma_\gen{x}$ 
  be $(s (1) \, s (2) \dotso s (n))$. Then the sign of a rectangle in $\eRect 
  (\gen{x}, \gen{y})$ depends only on $\sigma_\gen{x}$ and $\sigma_\gen{y}$.
\end{rmk}

The sign assignment from Theorem~\ref{thm:sign} is then used in \cite{MOST07} 
to construct a chain complex over $\Z$ as follows. The complex $\GCt (\grid) = 
\GCt (\grid; \Z)$ is the free $\Z$-module generated by elements of $\gen{S} 
(\grid)$; fixing a sign assignment $\sign$, the complex $\GCt(\grid)$ is 
endowed with the endomorphism $\bound_\sign \colon \GCt(\grid) \to 
\GCt(\grid)$, defined by
\[
\bound_\sign (\gen{x}) = \sum_{\gen{y} \in \gen{S} (\grid)} \sum_{\substack{r 
    \in \Rect^\circ (\gen{x}, \gen{y})\\ \Int (r) \cap \markers = \eset}} \sign 
(r) \cdot \gen{y} \in \GCt(\grid).
\]
One can then see that $(\GCt(\grid), \bound_\sign)$ is a chain complex.  
Indeed, the terms in $\bound_\sign \comp \bound_\sign (\gen{x})$ can be paired 
off as before, by the axioms defining $\sign$. Moreover, given $\sign_1$ and 
$\sign_2$, the map $\Phi \colon (\GCt(\grid), \bound_{\sign_1}) \to 
(\GCt(\grid), \bound_{\sign_2})$, defined by
\[
\Phi (\gen{x}) = g (\gen{x}) \cdot \gen{x},
\]
gives an isomorphism of the two chain complexes. Again, one can take the 
homology of the chain complex $(\GCt(\grid), \bound_\sign)$, and define
\[
  \HKt (\grid) = \Hom_* (\GCt(\grid), \bound_\sign).
\]
It is shown in \cite{MOST07} that
\[
  \HKt (\grid) \cong \HKh (\grid) \otimes V^{n - \ell}
\]
for some $\Z$-module $\HKh (\grid)$ that is a link invariant, where $V$ is a 
rank-$2$ free module over $\Z$, spanned by one generator in bigrading $(-1, 
-1)$ and another in bigrading $(0, 0)$. The link invariant $\HKh (\grid)$, also 
denoted by $\HFGh (L)$, is shown in \cite{Sar11} to be isomorphic to $\HFKh (L, 
\mathfrak{o})$ for some \emph{orientation system} $\mathfrak{o}$ of the link 
$L$.

We are now ready to turn to the proof of the analogous statement of 
Proposition~\ref{prop:exact}, with signs.

\begin{prop}
\label{prop:exact_Z}
There exists an exact triangle
\[
\dotsb \to \HKt (\grid_\infty; \Z) \to \HKt (\grid_0; \Z) \to \HKt (\grid_1; 
\Z) \to \dotsb
\]
\end{prop}

Our proof is reminiscent of that in Section~4 of \cite{MOST07}. We adopt the 
strategy from Section~\ref{sec:skein}; to do so, we must specify the signs used 
in defining our various chain maps and chain homotopies, and check that they 
indeed satisfy Lemma~\ref{lem:hom_alg}:

We begin by considering pentagons. First, we define the notion of a 
\emph{corresponding generator}. For each $\gen{x} \in C_k$, there exist exactly 
one $\gen{x'} \in C_{k+1}$ and one $\gen{x''} \in C_{k+2}$ that are canonically 
closest to $\gen{x}$; we require that $\gen{x, x'}$ and $\gen{x''}$ coincide 
everywhere except on the $\beta$ curves, and $\gen{x'}$ and $\gen{x''}$ are 
obtained from $\gen{x}$ by sliding the $\beta_k$-component horizontally to the 
$\beta_{k+1}$ and $\beta_{k+2}$ curves respectively. We define the maps $c_k^+ 
\colon C_k \to C_{k+1}$ and $c_k^- \colon C_k \to C_{k+2}$ by
\begin{align*}
  c_k^+ (\gen{x}) & = \gen{x'},\\
  c_k^- (\gen{x}) & = \gen{x''}.
\end{align*}

We can now define the straightening maps $d_k^\poly{P} \colon \ePent_k 
(\gen{x}, \gen{y}) \to \eRect_k (\gen{x}, c_k^- (\gen{y}))$ and $e_k^\poly{P} 
\colon \ePent_k (\gen{x}, \gen{y}) \to \eRect_{k+1} (c_k^+ (\gen{x}), \gen{y})$ 
as follows. Given $p \in \ePent_k (\gen{x}, \gen{y})$, we obtain $d_k^\poly{P} 
(p)$ by sliding the $\beta_{k+1}$-component of $\gen{y}$ back to the $\beta_k$ 
curve, thereby post-composing $p$ with a triangle; similarly, we obtain 
$e_k^\poly{P} (p)$ by sliding the $\beta_k$-component of $\gen{x}$ to the 
$\beta_{k+1}$ curve, thereby pre-composing $p$ with another triangle.  Notice 
that by Remark~\ref{rmk:sign}, we have
\[
  \sign (d_k^\poly{P} (p)) = \sign (e_k^\poly{P} (p)).
\]
We now define
\[
  \poly{P}_k (\gen{x}) = \sum_{\gen{y} \in \gen{S} (\grid_{k+1})} 
  \sum_{\substack{p \in \ePent_k (\gen{x}, \gen{y})\\ \Int (p) \cap \markers = 
      \eset}} \epsilon_k^\poly{P} (p) \cdot \gen{y} \in C_{k+1},
\]
where
\[
  \epsilon_k^\poly{P} (p) =
  \begin{cases}
    \sign (d_k^\poly{P} (p)) & \text{if } p \text{ is a left pentagon},\\
    - \sign (d_k^\poly{P} (p)) & \text{if } p \text{ is a right pentagon}.
  \end{cases}
\]

Turning to triangles, we again view generators as permutations. The signature 
$\sgn (\gen{x})$ of a generator $\gen{x}$ is defined to be the signature $\sgn 
(\sigma_\gen{x})$ of the corresponding permutation. Then we define
\[
  \poly{T}_k (\gen{x}) = \sum_{\gen{y} \in \gen{S} (\grid_{k+1})} 
  \sum_{\substack{p \in \Tri_k (\gen{x}, \gen{y})\\ \Int (p) \cap \markers = 
      \eset}} \epsilon_k^\poly{T} (p) \cdot \gen{y} \in C_{k+1},
\]
where
\[
  \epsilon_k^\poly{T} (p) = \sgn (\gen{x}).
\]

Finally, we define
\[
  f_k (\gen{x}) = \poly{P}_k (\gen{x}) + \poly{T}_k (\gen{x}) \in C_{k+1}.
\]

\begin{lem}
\label{lem:chain_Z}
The map $f_k$ is an anti-chain map. In fact, $\poly{P}_k$ and $\poly{T}_k$ are 
both anti-chain maps.
\end{lem}

\begin{proof}
The proof follows from the proof of Lemma~\ref{lem:chain}. We first consider 
the juxtaposition of a pentagon and a rectangle. If the two polygons are 
disjoint or have overlapping interiors, then the domain can be decomposed as 
either $r * p$ or $p' * r'$; then by Property~(1) of 
Definition~\ref{defn:sign}, $\sign (r) \cdot \sign (d_k^\poly{P} (p)) = - \sign 
(d_k^\poly{P} (p')) \cdot \sign (r')$, and consequently $\sign (r) \cdot 
\epsilon_k^\poly{P} (p) = - \epsilon_k^\poly{P} (p') \cdot \sign (r')$. If the 
two polygons share a corner, then the domain can be decomposed in two ways; 
straightening the pentagons (with either $d_k^\poly{P}$ or $e_k^\poly{P}$) and 
using Property~(1) of Definition~\ref{defn:sign}, we again see that the terms 
cancel.

Consider now the juxtaposition of a triangle $p$ and a rectangle $r$. Notice 
first that the differential $\bound$ always changes the signature of a 
generator; this means that $\sign (r) \cdot \epsilon_k^\poly{T} (p) = - 
\epsilon_k^\poly{T} (p') \cdot \sign (r')$. By Remark~\ref{rmk:tri_decomp}, 
such a domain can always be decomposed in two ways, one contributing to 
$\bound_{k+1} \comp \poly{T}_k$, and one to $\poly{T}_k \comp \bound_k$; 
moreover, the two rectangles involved correspond to the same permutation, and 
so in fact have the same sign.
\end{proof}

We similarly define the straightening maps $d_k^\poly{H} \colon \eHex_k 
(\gen{x}, \gen{y}) \to \eRect_k (\gen{x}, c_k^+ (\gen{y}))$ and $e_k^\poly{H} 
\colon \eHex_k (\gen{x}, \gen{y}) \to \eRect_{k+2} (c_k^- (\gen{x}), \gen{y})$ 
by sliding the appropriate component, and again notice that $\sign 
(d_k^\poly{H} (p)) = \sign (e_k^\poly{H} (p))$. We define
\[
  \poly{H}_k (\gen{x}) = \sum_{\gen{y} \in \gen{S} (\grid_{k+2})} 
  \sum_{\substack{p \in \eHex_k (\gen{x}, \gen{y})\\ \Int (p) \cap \markers = 
      \eset}} \epsilon_k^\poly{H} (p) \cdot \gen{y} \in C_{k+2},
\]
where
\[
  \epsilon_k^\poly{H} (p) = \sign (d_k^\poly{H} (p)).
\]

For quadrilaterals,
\[
  \poly{Q}_k (\gen{x}) = \sum_{\gen{y} \in \gen{S} (\grid_{k+2})} 
  \sum_{\substack{p \in \Quad_k (\gen{x}, \gen{y})\\ \Int (p) \cap \markers = 
      \eset}} \epsilon_k^\poly{Q} (p) \cdot \gen{y} \in C_{k+2},
\]
where
\[
  \epsilon_k^\poly{Q} (p) = \sgn (\gen{x}).
\]

\begin{lem}
\label{lem:cond_1_Z}
The maps $f_k$ and $\phi_k$ satisfy Condition~(1) of Lemma~\ref{lem:hom_alg}.
\end{lem}

\begin{proof}
Again the proof follows from that of Lemma~\ref{lem:cond_1}. We say that a 
domain $p$ is \emph{rectangle-like} if it is an allowed rectangle, pentagon, 
hexagon or heptagon, and we write $p \in \RL$; we say that it is 
\emph{triangle-like} if it is an allowed triangle or quadrilateral, and we 
write $p \in \TL$. We call a domain $p$ \emph{Type~I} if $p \in \RL * \RL$; 
\emph{Type~II} if $p \in \RL * \TL$, \emph{Type~III} if $p \in \TL * \RL$, and 
\emph{Type~IV} if $p \in \TL * \TL$. (Recall that $*$ and $\comp$ compose in 
the opposite order, so a term in $\poly{P}_{k+1} \comp \poly{T}_k$ is in $\TL * 
\RL$.) Refer to Table~\ref{tab:cond_1}. Typically, a domain can be decomposed 
in two ways; usually, it fits into one of these cases:
\begin{enumerate}
\item Both decompositions are Type~I. In this case, we see that the terms 
  cancel out by straightening the polygons and applying Property~(1) of 
  Definition~\ref{defn:sign};
\item One decomposition is Type~II, and one is Type~III. The domain is a left 
  domain. In this case, the terms cancel out because the two rectangle-like 
  polygons have the same sign, but the two triangle-like polygons have opposite 
  signs;
\item Both decompositions are Type~III. The domain is a right domain. In this 
  case, the two triangle-like polygons have the same sign. However, the two 
  rectangle-like polygons are both right domains, and one is a rectangle while 
  the other is a pentagon. Since right pentagons have opposite signs as the 
  straightened rectangles, the two rectangle-like polygons are of opposite 
  signs.
\end{enumerate}

The only special cases are those that involve two different domains, which are 
exactly the special cases in Lemma~\ref{lem:cond_1}.
\begin{enumerate}
  \item In this case, the two pentagons have the same sign, but the two 
    triangles have opposite signs;
\item In this case, the rectangle and the pentagon have the same sign, but the 
  triangle and the quadrilateral have opposite signs;
\item In one domain, the two triangles have the same sign, and so their 
  composite is positive; in the other domain, the composite of the straightened 
  rectangles is a vertical annulus, and so is negative by Property~(2) of 
  Definition~\ref{defn:sign}.
\end{enumerate}

Since there are no other cases, our proof is complete.
\end{proof}

Now for heptagons, there is only one straightening map $d_k^\poly{K} \colon 
\eHept_k (\gen{x}, \gen{y}) \to \eRect_k (\gen{x}, \gen{y})$. Seeing that the 
only allowed heptagons are right heptagons, we define
\[
  \poly{K}_k (\gen{x}) = \sum_{\gen{y} \in \gen{S} (\grid_k)} \sum_{\substack{p 
      \in \eHept_k (\gen{x}, \gen{y})\\ \Int (p) \cap \markers = \eset}} 
  \epsilon_k^\poly{K} (p) \cdot \gen{y} \in C_k,
\]
where
\[
  \epsilon_k^\poly{K} (p) = - \sign (d_k^\poly{K} (p)).
\]

\begin{lem}
\label{lem:cond_2_Z}
We have
\[
  \phi_{k+1} \comp f_k + f_{k+2} \comp \phi_k + \bound_k \comp \psi_k + \psi_k 
  \comp \bound_k = \id,
\]
so that the maps $f_k$ and $\phi_k$ satisfy Condition~(2) of 
Lemma~\ref{lem:hom_alg}.
\end{lem}

\begin{proof}
The typical cases are as in the proof of Lemma~\ref{lem:cond_1_Z}. We check the 
special cases in Lemma~\ref{lem:cond_2}.
\begin{enumerate}
\item This is actually a typical case; the two decompositions of the domain are 
  both Type~III;
\item In this case, the pentagons have the same sign, but the quadrilaterals 
  have opposite signs;
\item In this case, the pentagon and the hexagon have the same sign, but the 
  triangle and the quadrilateral have opposite signs;
\item Also in this case, the pentagon and the hexagon have the same sign, but 
  the triangle and the quadrilateral have opposite signs.
\end{enumerate}

We now check the decomposition of the identity map. If it is decomposed as a 
triangle and a quadrilateral, we see that they are of the same sign, and so we 
obtain a positive domain. If it is decomposed as a pentagon and a hexagon, we 
see that the composite of the straightened rectangles is negative by 
Property~(2) of Definition~\ref{defn:sign}, but the right pentagon and its 
straightening have opposite signs; this means that the overall domain is also 
positive.
\end{proof}

The proof of Proposition~\ref{prop:exact_Z} is completed by combining 
Lemmas~\ref{lem:hom_alg}, \ref{lem:chain_Z}, \ref{lem:cond_1_Z} and 
\ref{lem:cond_2_Z}.

%% file: sec_iterate.tex
\section{Iteration of the skein exact triangle}
\label{sec:iterate}

In this section, we will work only over $\F{2}$. Let $\grid_\infty, \grid_0, 
\grid_1$ be grid diagrams that are identical except near a point, as indicated 
in Figure~\ref{fig:grid}.  In Section~\ref{sec:skein}, we constructed the maps 
$f_k \colon \GCt (\grid_k) \to \GCt (\grid_{k+1})$ and $\phi_k \colon \GCt 
(\grid_k) \to \GCt (\grid_{k+2})$ that satisfy Lemma~\ref{lem:hom_alg}.  
Lemma~\ref{lem:hom_alg} and the Five Lemma together imply that $\GCt 
(\grid_\infty)$ is quasi-isomorphic to the mapping cone of $f_0 \colon \GCt 
(\grid_0) \to \GCt (\grid_1)$ (see \cite[Lemma~4.2]{OS05}), where the 
quasi-isomorphism is given by $f_\infty + \phi_\infty \colon \GCt 
(\grid_\infty) \to \GCt (\grid_0) \oplus \GCt (\grid_1)$.

We now wish to iterate this quasi-isomorphism to obtain a cube of resolutions 
that computes the same homology.

\begin{figure}
\includegraphics[width=.45\columnwidth]{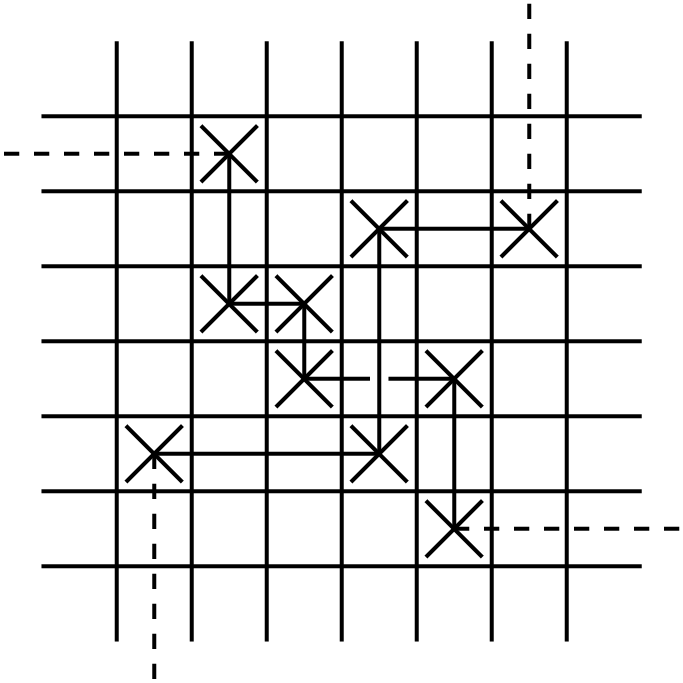}
\caption{The grid diagram $\grid$ of an oriented link $L$ near a crossing. The 
  $6 \times 6$ block of cells in the center is the block associated to this 
  crossing.}
\label{fig:block}
\end{figure}

Let the crossings of a link $L$ be numbered from $1$ to $m$. Start with a 
planar projection of $L$, and convert it into a grid diagram. By applying 
stabilisation and commutation as described in \cite{Cro95, Dyn06, MOST07}, we 
can require that
\begin{enumerate}
\item Near every crossing, the diagram is as indicated in 
  Figure~\ref{fig:block}. For the $i$-th crossing, the associated $6 \times 6$ 
  block of cells illustrated is referred to as the $i$-th block;
\item If $i \neq j$, then the $i$-th block and the $j$-th block occupy disjoint 
  rows and disjoint columns.
\end{enumerate}
Let the resulting grid diagram be $\grid$. To each sequence $k_1, k_2, \dotsc, 
k_m$, where $k_i \in \set{\infty, 0, 1}$ for $1 \leq i \leq m$, we associate a 
grid diagram $\grid_{k_1, \dotsc, k_m}$.  The diagram $\grid_{k_1, \dotsc, 
  k_m}$ is obtained from $\grid$ by replacing the $i$-th block by the 
appropriate $6 \times 6$ block as in Figure~\ref{fig:grid} (where the central 
$4 \times 4$ block is shown), depending on the value of $k_i$. In particular, 
$\grid_{\infty, \dotsc, \infty} = \grid$. Let $C_{k_1, \dotsc, k_m}$ be the 
associated chain complex of $\grid_{k_1, \dotsc, k_m}$, equipped with the 
differential map. For $1 \leq i \leq m$, we let the edge map
\[
  f_{k_1, \dotsc, k_m}^i \colon C_{k_1, \dotsc, k_{i-1}, k_i, k_{i+1}, \dotsc, 
    k_m} \to C_{k_1, \dotsc, k_{i-1}, k_i + 1, k_{i+1}, \dotsc, k_m}
\]
be the map $f_k$ defined in Section~\ref{sec:skein}; this makes sense, since 
the two chain complexes differ only near a crossing.  Analogously, we have the 
map
\[
  \phi_{k_1, \dotsc, k_m}^i \colon C_{k_1, \dotsc, k_{i-1}, k_i, k_{i+1}, 
    \dotsc, k_m} \to C_{k_1, \dotsc, k_{i-1}, k_i + 2, k_{i+1}, \dotsc, k_m},
\]
which is just the map $\phi_k$ defined earlier.

This allows us to define the \emph{big cube of resolutions} of $\grid$ to be 
the complex
\begin{align*}
  (\BCR (\grid), \bound_{\BCR}) & = \left(\bigoplus_{k_i \in \set{\infty, 0, 
        1}} C_{k_1, \dotsc, k_m}, \sum_{k_i \in \set{\infty, 0, 1}} 
  \right(\bound_{k_1, \dotsc, k_m}\\
  & \left. \left. \qquad + \sum_{j \text{ s.t. } k_j \in \set{\infty, 0}} 
      f_{k_1, \dotsc, k_m}^j + \sum_{t \text{ s.t. } k_t = \infty} \phi_{k_1, 
        \dotsc, k_m}^t \right) \right),
\end{align*}
and the \emph{(small) cube of resolutions} of $\grid$ to be the complex
\[
  (\CR (\grid), \bound_{\CR}) = \brac{\bigoplus_{k_i \in \set{0, 1}} C_{k_1, 
      \dotsc, k_m}, \sum_{k_i \in \set{0, 1}} \brac{\bound_{k_1, \dotsc, k_m} + 
      \sum_{j \text{ s.t. } k_j = 0} f_{k_1, \dotsc, k_m}^j}},
\]
which is a subcomplex of $\BCR (\grid)$. In the cube of resolutions $\CR 
(\grid)$, each vertex is associated to a grid diagram in which all crossings 
have been resolved. The case when there are two crossings is illustrated below:
\[
  \xymatrix{
    C_{\infty,\infty} \ar[r]_{f_{\infty,\infty}^2} 
    \ar@/^1pc/[rr]^{\phi_{\infty,\infty}^2}
    & C_{\infty,0} \ar[r]_{f_{\infty,0}^2} \ar[d]^{f_{\infty,0}^1} 
    \ar@/_1pc/[dd]_{\phi_{\infty,0}^1}
    & C_{\infty,1} \ar[d]_{f_{\infty,1}^1} \ar@/^1pc/[dd]^{\phi_{\infty,1}^1}\\
    & C_{0,0} \ar[r]_{f_{0,0}^2} \ar[d]^{f_{0,0}^1}
    & C_{0,1} \ar[d]_{f_{0,1}^1}\\
    & C_{1,0} \ar[r]_{f_{1,0}^2}
    & C_{1,1}
  }
\]
The cube of resolutions $\CR (\grid)$ consists of the $2 \times 2$ square on 
the lower right. The big cube of resolutions $\BCR (\grid)$ consists of the 
whole diagram, together with $C_{0,\infty}, C_{1,\infty}, f_{\infty,\infty}^1, 
f_{0,\infty}^1, f_{0,\infty}^2, f_{1,\infty}^2, \phi_{\infty,\infty}^1, 
\phi_{0,\infty}^2, \phi_{1,\infty}^2$, which are not shown. (These would all be 
at the lower left of the diagram.)

\begin{rmk}
  Neither $\BCR (\grid)$ nor $\CR (\grid)$ contains diagonal maps, e.g.\ a map 
  that goes from $C_{\infty,0}$ to $C_{0,1}$ or $C_{1,1}$. This is in stark 
  contrast with cubes of resolutions in many other contexts. For example, in 
  \cite{OS05}, where the technique of spectral sequences is first applied to 
  Heegaard Floer homology, there are diagonal maps in both the big 
  $\set{\infty, 0, 1}^m$ cube and the small $\set{0, 1}^m$ cube.  Such diagonal 
  maps are needed because the edge maps $f$ and $\phi$ only commute \emph{up to 
    chain homotopy}. Below, Lemma~\ref{lem:iterate} will guarantee that our 
  edge maps commute \emph{on the nose}, allowing us to define the diagonal maps 
  to be zero.
\end{rmk}

In the case with two crossings, observe that $C_{\infty,\infty}$ is 
quasi-isomorphic to the mapping cone of $f_{\infty,0}^2$ via 
$f_{\infty,\infty}^2 + \phi_{\infty,\infty}^2$. The fact that the diagram 
commutes now immediately implies that
\begin{enumerate}
  \item the cube of resolutions $\CR (\grid)$ is a chain complex;
  \item the sum $f_{\infty,0}^1 + f_{\infty,1}^1 + \phi_{\infty,0}^1 + 
    \phi_{\infty,1}^1$ is a chain map from the mapping cone of $f_{\infty,0}^2$ 
    to the cube of resolutions $\CR (\grid)$; and
  \item this chain map is a quasi-isomorphism.
\end{enumerate}
(To see the last statement, observe that if $f \colon C_1 \to C_2$ and $f' 
\colon C_1' \to C_2'$ are quasi-isomorphisms, and if there exist maps $g_1 
\colon C_1 \to C_1'$ and $g_2 \colon C_2 \to C_2'$ so that the diagram
\[
  \xymatrix{
    C_1 \ar[r]^{f} \ar[d]_{g_1} & C_2 \ar[d]^{g_2}\\
               C_1' \ar[r]^{f'} & C_2'
  }
\]
commutes, then $f + f'$ is a quasi-isomorphism between the mapping cone of 
$g_1$ and that of $g_2$.) Thus, we see that $C_{\infty,\infty}$ is 
quasi-isomorphic to $\CR (\grid)$.

The general case is similar; to be precise, our claim is the following.

\begin{prop}
\label{prop:iterate}
The cube of resolutions $\BCR (\grid; \F{2})$ and $\CR (\grid; \F{2})$ are 
indeed chain complexes. Moreover, $\CR (\grid; \F{2})$ is quasi-isomorphic to 
$\GCt (\grid; \F{2})$. Consequently, $\BCR (\grid; \F{2})$ is acyclic.
\end{prop}

As mentioned, the main ingredient in proving Proposition~\ref{prop:iterate} is 
the following lemma.

\begin{lem}
\label{lem:iterate}
All maps involved commute. Precisely,
\begin{align*}
  f_{k_1, \dotsc, k_{i_1} + 1, \dotsc, k_m}^{i_2} \comp f_{k_1, \dotsc, 
    k_{i_1}, \dotsc, k_m}^{i_1} & = f_{k_1, \dotsc, k_{i_2} + 1, \dotsc, 
    k_m}^{i_1} \comp f_{k_1, \dotsc, k_{i_2}, \dotsc, k_m}^{i_2}.\\
  \phi_{k_1, \dotsc, k_{i_1} + 1, \dotsc, k_m}^{i_2} \comp f_{k_1, \dotsc, 
    k_{i_1}, \dotsc, k_m}^{i_1} & = f_{k_1, \dotsc, k_{i_2} + 2, \dotsc, 
    k_m}^{i_1} \comp \phi_{k_1, \dotsc, k_{i_2}, \dotsc, k_m}^{i_2}.
\end{align*}
\end{lem}

\begin{proof}
Inspecting the $6 \times 6$ block in Figure~\ref{fig:block}, we see that there 
is an $X$ near each corner of the block, and so each allowed polygon defined in 
Section~\ref{sec:skein} counted in a chain map or a chain homotopy can only 
leave the block either horizontally or vertically. If it leaves the block 
horizontally, then it is contained in the rows that the block spans; if it 
leaves the block vertically, then it is contained in the columns that the block 
spans. This shows that if two polygons from two different crossings share a 
corner, one must be long and horizontal, and the other long and vertical; in 
such cases, the composite domain always has an obvious alternative 
decomposition. The cases where the two polygons are disjoint or have 
overlapping interiors are obvious.
\end{proof}

\begin{proof}[Proof of Proposition~\ref{prop:iterate}]
Exactly as in the case with two crossings, Lemma~\ref{lem:iterate} implies that 
the cube of resolutions is indeed a chain complex, and that all appropriate 
maps are chain maps. We proceed by induction; at each step, we claim that the 
chain complexes
\begin{align*}
  & \left(\bigoplus_{k_i \in \set{0, 1}} C_{\infty, \dotsc, \infty, k_{t+1}, 
      \dotsc, k_m}, \right.\\
  & \quad \left.\sum_{k_i \in \set{0, 1}} \left(\bound_{\infty, \dotsc, \infty, 
        k_{t+1}, \dotsc, k_m} + \sum_{j \text{ s.t. } k_j = 0} f_{\infty, 
        \dotsc, \infty, k_{t+1}, \dotsc, k_m}^j\right)\right)
\end{align*}
and
\[
  \brac{\bigoplus_{k_i \in \set{0, 1}} C_{\infty, \dotsc, \infty, k_t, \dotsc, 
      k_m}, \sum_{k_i \in \set{0, 1}} \brac{\bound_{\infty, \dotsc, \infty, 
        k_t, \dotsc, k_m} + \sum_{j \text{ s.t. } k_j = 0} f_{\infty, \dotsc, 
        \infty, k_t, \dotsc, k_m}^j}}
\]
are quasi-isomorphic. The quasi-isomorphism is given by
\[
  \sum_{k_i \in \set{0, 1}} f_{\infty, \dotsc, \infty, k_{t+1}, \dotsc, k_m}^t 
  + \sum_{k_i \in \set{0, 1}} \phi_{\infty, \dotsc, \infty, k_{t+1}, \dotsc, 
    k_m}^t.
\]
This map is a quasi-isomorphism by the induction hypothesis and by the comment 
after (3) in the case with two crossings above.
\end{proof}

Corollary~\ref{cor:main} follows from Proposition~\ref{prop:iterate}.

\begin{rmk}
  \label{rmk:iterateZ}
  Over $\Z$, for the cube of resolutions to be a chain complex, the analogue of 
  Lemma~\ref{lem:iterate} over $\Z$ should presumably state that the maps 
  involved \emph{anti-commute}. However, if we denote by $\rho^i$ (resp.\ 
  $\tau^i$) the maps defined by the rectangle-like (resp.\ triangle-like) 
  polygons associated to the $i$-th crossing following the definitions in 
  Section~\ref{sec:signs}, then we have
  \begin{enumerate}
    \item $\rho^2 \circ \rho^1 = - \rho^1 \circ \rho^2$;
    \item $\rho^2 \circ \tau^1 = - \tau^1 \circ \rho^2$;
    \item $\tau^2 \circ \rho^1 = - \rho^1 \circ \tau^2$;
    \item $\tau^2 \circ \tau^1 = \tau^1 \circ \tau^2$.
  \end{enumerate}
  Since \emph{only} the maps defined by triangle-like polygons commute, the 
  author has thus far been unable to prove a version of 
  Lemma~\ref{lem:iterate}, and therefore Proposition~\ref{prop:iterate}, over 
  $\Z$.
\end{rmk}

%% file: sec_quasi.tex
\section{Quasi-alternating links}
\label{sec:quasi}

We now describe an application of the skein exact triangle. To begin, we must 
first define the $\delta$-grading on knot Floer homology. The $\delta$-grading 
is closely related to the Maslov and Alexander gradings; in view of that, in 
this section we revert to the traditional notation with both $O$'s and $X$'s, 
and denote the set of $O$'s by $\markersO$ and that of $X$'s by $\markers$.

The following formulation is found in \cite{MOST07}. Given two collections $A$ 
and $B$ of finitely many points in the plane, let
\begin{align*}
  \countJ (A, B) = & \frac{1}{2} \# \setc{((a_1, a_2), (b_1, b_2)) \in A \times 
    B}{a_1 < b_1 \text{ and } a_2 < b_2}\\
  & + \frac{1}{2} \# \setc{(a_1, a_2), (b_1, b_2)) \in A \times B}{b_1 < a_1 
    \text{ and } b_2 < a_2}.
\end{align*}
Treating $\gen{x} \in \gen{S} (\grid)$ as a collection of points with integer 
coordinates in a fundamental domain for $\torus$, and similarly $\markersO$ and 
$\markers$ as collections of points in the plane with half-integer coordinates, 
the Maslov grading of a generator is given by
\[
  M (\gen{x}) = \countJ (\gen{x}, \gen{x}) - 2 \countJ (\gen{x}, \markersO) + 
  \countJ (\markersO, \markersO) + 1,
\]
while the Alexander grading is given by
\[
  A (\gen{x}) = \countJ (\gen{x}, \markers) - \countJ (\gen{x}, \markersO) - 
  \frac{1}{2} \countJ (\markers, \markers) + \frac{1}{2} \countJ (\markersO, 
  \markersO) - \frac{n - 1}{2},
\]
where $n$ is the size of the grid diagram. Now the $\delta$-grading is just
\[
  \delta (\gen{x}) = A (\gen{x}) - M (\gen{x});
\]
in other words, we have
\[
  \delta (\gen{x}) = - \countJ (\gen{x}, \gen{x}) + \countJ (\gen{x}, \markers) 
  + \countJ (\gen{x}, \markersO) - \frac{1}{2} \countJ (\markers, \markers) - 
  \frac{1}{2} \countJ (\markersO, \markersO) - \frac{n + 1}{2}.
\]
These gradings do not depend on the choice of the fundamental domain; moreover, 
they agree with the original definitions in terms of pseudo-holomorphic 
representatives.

It has been observed that, for many classes of links, the knot Floer homology 
over $R = \F{2}$ or $\Z$ is a free $R$-module supported in only one 
$\delta$-grading, which motivates the following definition \cite{Ras03, Ras05, 
  MO08}:

\begin{defn}
Let $R = \F{2}$ or $\Z$. A link $L$ is \emph{Floer-homologically thin} over $R$ 
if $\HFKh (L; R)$ is a free $R$-module supported in only one $\delta$-grading.  
If in addition the $\delta$-grading equals $- \sigma (L) / 2$, where $\sigma 
(L)$ is the signature of the link, then we say that $L$ is 
\emph{Floer-homologically $\sigma$-thin}.
\end{defn}

In \cite{MO08}, Manolescu and Ozsv\'{a}th show that a class of links, which is 
a natural generalisation of alternating knots, has the homologically 
$\sigma$-thin property over $\F{2}$. Precisely, we recall the following 
definition from \cite{OS05}:

\begin{defn}
The set of \emph{quasi-alternating links} $\quasi$ is the smallest set of links 
satisfying the following properties:
\begin{enumerate}
  \item The unknot is in $\quasi$.
  \item If $L_\infty$ is a link that admits a projection with a crossing such 
    that
    \begin{enumerate}
      \item both resolutions $L_0$ and $L_1$ at that crossing are in $\quasi$; 
        and
      \item $\det (L_\infty) = \det (L_0) + \det (L_1)$,
    \end{enumerate}
    then $L_\infty$ is in $\quasi$.
\end{enumerate}
\end{defn}

The result of Manolescu and Ozsv\'{a}th is then the following statement:

\begin{thm}[Manolescu--Ozsv\'{a}th]
\label{thm:thin}
Quasi-alternating links are Floer-ho\-mo\-log\-i\-cal\-ly $\sigma$-thin over 
$\F{2}$.
\end{thm}

The proof of the theorem is essentially an application of Manolescu's 
unoriented skein exact triangle; the main work is in tracking the changes in 
the $\delta$-grading of the maps involved. The same idea works in the current 
context as well: With grid diagrams, we may similarly track the changes in the 
$\delta$-grading, defined by the formula above. Of course, in our case, we will 
be proving the result for $\HFGh$ of a link instead of $\HFKh$, as in 
Theorem~\ref{thm:delta_sigma}. From this, we will obtain 
Theorem~\ref{thm:thinZ}, a strengthened version of Theorem~\ref{thm:thin}.

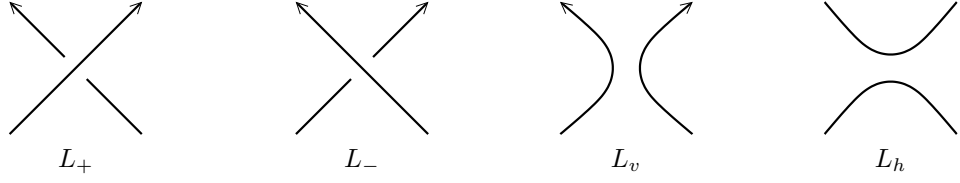
\begin{figure}
  \resizebox{\textwidth}{!}{\input{fig_links2.pstex_t}}
  \caption{$L_+, L_-, L_h$ and $L_v$ near a point.}
  \label{fig:links2}
\end{figure}

To begin, fix a crossing $c_0$ in the planar diagram of a link $L_\infty$. Let 
$L_+$ be the link with a positive crossing at $c_0$, and $L_-$ the link with a 
negative crossing; then either $L_\infty = L_+$ or $L_\infty = L_-$. Let $L_h$ 
and $L_v$ be the unoriented and oriented resolutions of $L_\infty$ at $c_0$ 
respectively, and choose an arbitrary orientation for $L_h$. This is 
illustrated in Figure~\ref{fig:links2}. Comparing with Figure~\ref{fig:links},  
if $L_\infty = L_+$, then $L_0 = L_h$ and $L_1 = L_v$; if instead $L_\infty = 
L_-$, then $L_0 = L_v$ and $L_1 = L_h$.

Denote by $D_+, D_-, D_v, D_h$ the planar diagrams of $L_+, L_-, L_v, L_h$, 
differing from each other only at $c_0$.  The following lemma is used in 
\cite{MO08}; for $L_+$, the first equality is proven by Murasugi in 
\cite{Mur65}, while the second is also inspired by a result of Murasugi from 
\cite{Mur70}. The equalities for $L_-$ are similarly obtained.

\begin{lem}
\label{lem:sigma}
Suppose that $\det (L_v), \det (L_h) > 0$. Let $e$ denotes the difference 
between the number of negative crossings in $D_h$ and the number of such 
crossings in $D_+$. If $\det (L_+) = \det (L_v) + \det (L_h)$, then
\begin{enumerate}
  \item $\sigma (L_v) - \sigma (L_+) = 1$; and
  \item $\sigma (L_h) - \sigma (L_+) = e$,
\end{enumerate}
If $\det (L_-) = \det (L_v) + \det (L_h)$, then
\begin{enumerate}
  \item $\sigma (L_v) - \sigma (L_-) = - 1$; and
  \item $\sigma (L_h) - \sigma (L_-) = e$.
\end{enumerate}
\end{lem}

Now we investigate the changes in the $\delta$-grading in the maps $f_k$ 
defined in Sections~\ref{sec:skein} and \ref{sec:signs}.

Given two grid diagrams $\grid, \grid'$ of the same size $n$, we can think of 
them as being on the same grid (i.e.\ consisting of the same horizontal and 
vertical circles) but having different $O$'s and $X$'s; therefore, we can write 
$\grid = (n, \markersO, \markers)$ and $\grid' = (n, \markersO', \markers')$.  
This allows us to identify $\gen{S} (\grid)$ and $\gen{S} (\grid')$, by viewing 
each generator $\gen{x} \in \gen{S} (\grid)$ as the permutation 
$\sigma_\gen{x}$ referred to in Remark~\ref{rmk:sign}. In this point of view, 
we will write $\delta_\grid (\gen{x})$ and $\delta_{\grid'} (\gen{x})$ to 
denote the gradings defined by applying the $\delta$-grading formula to 
$(\markersO, \markers)$ and $(\markersO', \markers')$ respectively.

In particular, for the rest of this section, we will view $\grid_\infty, 
\grid_0, \grid_1$ in this manner.

\begin{lem}
\label{lem:delta}
Let $\gen{x} \in \gen{S} (\grid)$ be a fixed generator in a grid diagram $\grid 
= (n, \markersO, \markers)$.  Then the following is true:
\begin{enumerate}
  \item If there is an empty rectangle $r$ from $\gen{x}$ to $\gen{y}$, 
    possibly containing $O$'s and $X$'s, then $\delta_{\grid} (\gen{y}) = 
    \delta_{\grid} (\gen{x}) + 1 - \# (\Int (r) \cap (\markersO \cup 
    \markers))$.
  \item Suppose $\grid' = (n, \markersO', \markers)$ is a grid diagram 
    identical to $\grid$ except in two adjacent columns, where the horizontal 
    positions of a pair of $O$ markers are interchanged, as in 
    Figure~\ref{fig:markers}. Then
    \begin{enumerate}
      \item $\delta_{\grid'} (\gen{x}) - \delta_{\grid} (\gen{x}) = -1/2$ if, 
        in $\grid$, the component of $\gen{x}$ on the vertical circle between 
        the markers lies to the northeast of one marker and to the southwest of 
        the other (the component of $\gen{x}$ indicated in the diagrams by a 
        solid point); and
      \item $\delta_{\grid'} (\gen{x}) - \delta_{\grid} (\gen{x}) = 1/2$ 
        otherwise (the component of $\gen{x}$ indicated by a hollow point).
    \end{enumerate}
    The same statement holds if, instead, the horizontal positions of a pair of 
    $X$ markers are interchanged, with $\grid' = (n, \markersO, \markers')$.
  \item If the diagram $\grid'$ is obtained from $\grid$ by reversing the 
    orientation of a component $K$ of $L$, then $\delta_{\grid'} (\gen{x}) - 
    \delta_{\grid} (\gen{x}) = -\epsilon/2$, where $\epsilon$ is the difference 
    between the number of negative crossings in $\grid'$ and the number of such 
    crossings in $\grid$.
\end{enumerate}
\end{lem}

\begin{figure}
  \resizebox{0.5\textwidth}{!}{\input{fig_markers.pstex_t}}
  \caption{Moving two markers across a vertical circle. Observe that before the 
    move, the solid point lies to the northeast of the $O$ and to the southwest 
    of the $X$, thus forming two southwest--northeast pairs with the markers 
    that are destroyed in the move. The hollow point forms the same number of 
    pairs with the markers before and after the move.}
  \label{fig:markers}
\end{figure}
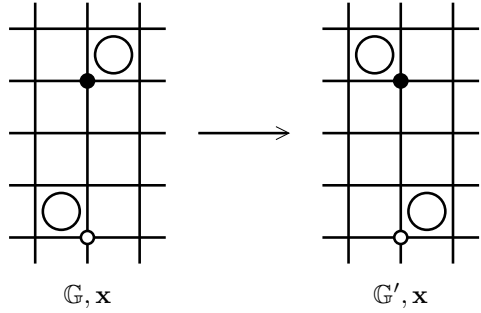

\begin{proof}
Recall that
\[
  \delta (\gen{x}) = - \countJ (\gen{x}, \gen{x}) + \countJ (\gen{x}, \markers) 
  + \countJ (\gen{x}, \markersO) - \frac{1}{2} \countJ (\markers, \markers) - 
  \frac{1}{2} \countJ (\markersO, \markersO) - \frac{n + 1}{2}.
\]
Observe that $\countJ (A, B)$ is the number of southwest--northeast pairs 
between $A$ and $B$, divided by two. Consider the following:
\begin{enumerate}
  \item Observe that the last three terms in $\delta (\gen{x})$ and $\delta 
    (\gen{y})$ are identical. Switching from $\gen{x}$ to $\gen{y}$ destroys 
    exactly one southwest--northeast pair, which is counted twice and 
    contributes $+1$ to $\countJ (\gen{x}, \gen{x})$. For every $O$ or $X$ 
    inside $\Int (r)$, switching from $\gen{x}$ to $\gen{y}$ destroys two 
    southwest--northeast pairs, each counted once, that contribute $+1$ to 
    $\countJ (\gen{x}, \markers) + \countJ (\gen{x}, \markersO)$.
  \item Consider the first case, where $\grid = (n, \markersO, \markers)$ and 
    $\grid' = (n, \markersO', \markers)$, as in Figure~\ref{fig:markers}. We 
    first see that $\countJ (\markersO', \markersO') - \countJ (\markersO, 
    \markersO) = - 1$.  If the component of $\gen{x}$ on the vertical circle 
    lies to the northeast of one marker and to the southwest of the other, then 
    $\countJ (\gen{x}, \markersO') - \countJ (\gen{x}, \markersO) = -1$; 
    otherwise, this quantity is $0$. All other terms in $\delta_{\grid} 
    (\gen{x})$ and $\delta_{\grid'} (\gen{x})$ are identical. The second case 
    is entirely analogous.
  \item This is simply a restatement of \cite[Proposition~5.4]{MOST07}.
\end{enumerate}
This completes the proof of the lemma.
\end{proof}

The following is an easy consequence of Lemma~\ref{lem:delta}~(1) and (2):

\begin{lem}
  \label{lem:orient_delta}
  Suppose $\grid_k$ and $\grid_{k+1}$ are grid presentations of two links with 
  compatible orientations (so that $\grid_k$ and $\grid_{k+1}$ have compatible 
  sets of $O$'s), and $f_k \colon \grid_k \to \grid_{k+1}$ is the map defined 
  in Section~\ref{sec:skein} and Section~\ref{sec:signs}. If $\gen{x} \in 
  \gen{S} (\grid_k)$, and $\gen{y} \in \gen{S} (\grid_{k+1})$ appears in $f_k 
  (\gen{x})$, then $\delta_{\grid_{k+1}} (\gen{y}) - \delta_{\grid_k} (\gen{x}) 
  = +1/2$.
\end{lem}

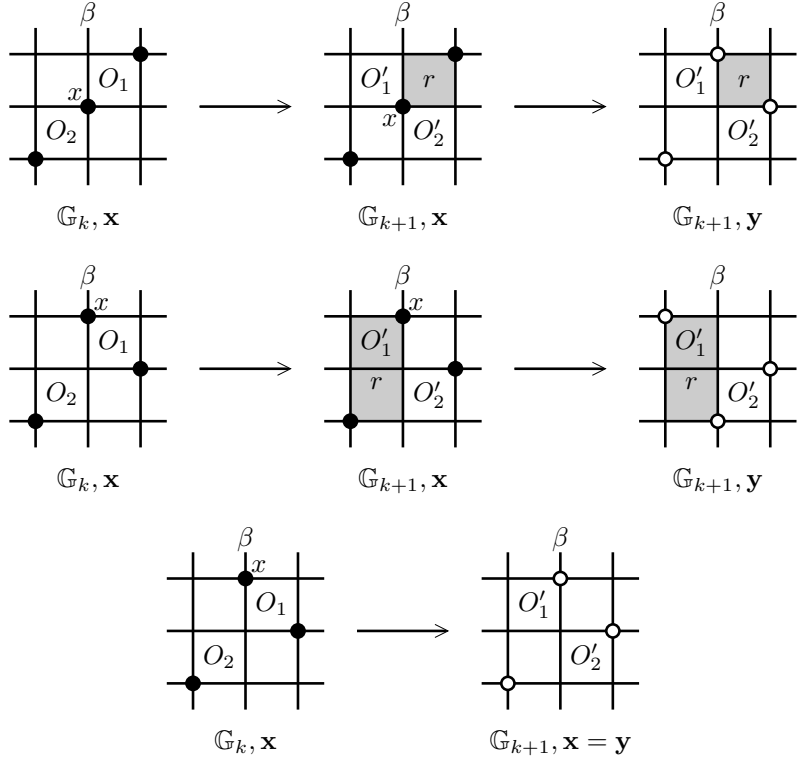
\begin{figure}
  \resizebox{0.833\textwidth}{!}{\input{fig_orient_delta.pstex_t}}
  \caption{Computing the $\delta$-grading change under $f_k$ when $\grid_k$ and 
    $\grid_{k+1}$ have compatible orientations. The top and middle rows show 
    the case when the polygon $p$ being counted is a pentagon, while the bottom 
    row shows the case when $p$ is a triangle. In both cases, $\gen{y}$ appears 
    in $f_k (\gen{x})$. Note that there may be multiple horizontal circles 
    between the $O$'s, which are omitted for the sake of simplicity, in each of 
    the figures above.}
  \label{fig:orient_delta}
\end{figure}

\begin{proof}
  Note first that the grid diagrams $\grid_k = (n, \markersO, \markers)$ and 
  $\grid_{k+1} = (n, \markersO', \markers')$ are related as in 
  Lemma~\ref{lem:delta}~(2).  Refer to Figure~\ref{fig:orient_delta}. Without 
  loss of generality, assume that the two relevant markers are $O$'s. (The case 
  where they are $X$'s is completely analogous.) Label the two markers in 
  $\markersO$ by $O_1$ and $O_2$, such that $O_1$ lies to the northeast of 
  $O_2$; similarly, label the markers in $\markersO'$ by $O_1'$ and $O_2'$, 
  such that $O_1'$ lies to the northwest of $O_2'$. In the current framework, 
  $\beta_\infty, \beta_0$ and $\beta_1$ are all represented by the vertical 
  circle $\beta$ in the middle. Let $x$ be the $\beta$-component of $\gen{x}$.  
  
  Suppose $\gen{y} \in \gen{S} (\grid_{k+1})$ appears in $f_k (\gen{x})$. Then
  \[
    \delta_{\grid_{k+1}} (\gen{y}) - \delta_{\grid_k} (\gen{x}) = 
    (\delta_{\grid_{k+1}} (\gen{x}) - \delta_{\grid_k} (\gen{x})) + 
    (\delta_{\grid_{k+1}} (\gen{y}) - \delta_{\grid_{k+1}} (\gen{x})).
  \]
  There are two cases, as follows.
  
  Suppose there is an empty pentagon $p$ from $\gen{x}$ to $\gen{y}$; in the 
  framework described in this section, it can in fact be viewed as an empty 
  rectangle $r$ from $\gen{x}$ to $\gen{y}$ in $\grid_{k+1}$. Observe that the 
  horizontal circles divide $\torus$ into a number of components; let $A$ be 
  the unique component containing both $O_1$ in $\grid_k$ and $O_1'$ in 
  $\grid_{k+1}$. Since the boundary of $p$ contains $u_k$ (cf.\ 
  Figure~\ref{fig:one_grid}), which lies in $A$, the boundary of $r$ must 
  contain $\beta \cap A$. Therefore, one of two scenarios must be true:
  \begin{enumerate}
    \item The point $x$ lies to the northeast of $O_2$ and to the southwest of 
      $O_1$. Then $\Int (r) \cap (\markersO \cup \markers) = \eset$. By 
      Lemma~\ref{lem:delta}~(2)(a), $\delta_{\grid_{k+1}} (\gen{x}) - 
      \delta_{\grid_k} (\gen{x}) = -1/2$, and by Lemma~\ref{lem:delta}~(1), 
      $\delta_{\grid_{k+1}} (\gen{y}) - \delta_{\grid_{k+1}} (\gen{x}) = 1$.  
      This is illustrated by the top row of Figure~\ref{fig:orient_delta}.
    \item The point $x$ lies elsewhere. Then $\Int (r) \cap (\markersO \cup 
      \markers)$ is either $\set{O_1'}$ or $\set{O_2'}$. By 
      Lemma~\ref{lem:delta}~(2)(b), $\delta_{\grid_{k+1}} (\gen{x}) - 
      \delta_{\grid_k} (\gen{x}) = +1/2$, and by Lemma~\ref{lem:delta}~(1), 
      $\delta_{\grid_{k+1}} (\gen{y}) - \delta_{\grid_{k+1}} (\gen{x}) = 1-1 = 
      0$. This is illustrated by the middle row of 
      Figure~\ref{fig:orient_delta}.
  \end{enumerate}
  Anyway, we get that $\delta_{\grid_{k+1}} (\gen{y}) - \delta_{\grid_k} 
  (\gen{x}) = +1/2$.

  Suppose there is a triangle $p$ from $\gen{x}$ to $\gen{y}$; then, in the 
  framework of this section, $\gen{x} = \gen{y}$, and so $\delta_{\grid_{k+1}} 
  (\gen{y}) - \delta_{\grid_{k+1}} (\gen{x}) = 0$.  Since the boundary of $p$ 
  contains $v_k$ (cf.\ Figure~\ref{fig:one_grid}), we see that $\Int (p) 
  \subset \Int (t_k) \cup \Int (t_{k+1})$. This implies that $x$ cannot lie 
  both to the northeast of $O_2$ and to the southwest of $O_1$.  Thus, by 
  Lemma~\ref{lem:delta}~(2)(b), $\delta_{\grid_{k+1}} (\gen{x}) - 
  \delta_{\grid_k} (\gen{x}) = +1/2$. This case is illustrated by the bottom 
  row of Figure~\ref{fig:orient_delta}.
\end{proof}

We can now prove a graded version of Theorem~\ref{thm:main}:

\begin{prop}
\label{prop:delta_main}
With respect to the $\delta$-grading, the skein exact sequence in 
Theorem~\ref{thm:main} can be written as
\begin{align*}
  \dotsb & \to \HFGh_{*-\frac{1}{2}} (L_v; R) \otimes V^{n-\ell_v} \to \HFGh_* 
  (L_+; R) \otimes V^{n-\ell_+}\\
  & \to \HFGh_{*-\frac{e}{2}}  (L_h; R) \otimes V^{n-\ell_h} \to 
  \HFGh_{*-\frac{1}{2}+1} (L_v; R) \otimes V^{n-\ell_v} \to \dotsb
\end{align*}
and
\begin{align*}
  \dotsb & \to \HFGh_{*-\frac{e}{2}} (L_h; R) \otimes V^{n-\ell_h} \to \HFGh_* 
  (L_-; R) \otimes V^{n-\ell_-}\\
  & \to \HFGh_{*+\frac{1}{2}}  (L_v; R) \otimes V^{n-\ell_v} \to 
  \HFGh_{*-\frac{e}{2}+1} (L_h; R) \otimes V^{n-\ell_h} \to \dotsb,
\end{align*}
where $R = \F{2}$ or $\Z$, $V$ is a free module of rank $2$ over $R$ with 
grading zero, and $e$ is as in the statement of Lemma~\ref{lem:sigma}.
\end{prop}

\begin{proof}
  Suppose $L_\infty$ has a positive crossing at $c_0$, so that $L_+ = 
  L_\infty$, $L_v = L_1$ and $L_h = L_0$. Let $\grid_\infty, \grid_0, \grid_1$ 
  be grid diagrams for $L_\infty, L_0, L_1$ respectively, differing with each 
  other only at $c_0$ as in Figure~\ref{fig:grid}. Note that in 
  Figure~\ref{fig:grid} only $X$'s are used as markers, whereas in the present 
  context both $X$'s and $O$'s are used.  Then we are to prove that in the 
  exact sequence
  \[
    \dotsb \to \HKt (\grid_1) \xrightarrow{(f_1)_*} \HKt (\grid_\infty) 
    \xrightarrow{(f_\infty)_*}  \HKt (\grid_0) \xrightarrow{(f_0)_*} \HKt 
    (\grid_1) \to \dotsb,
  \]
  the map $f_1$ shifts the $\delta$-grading by $+1/2$, $f_\infty$ by $-e/2$, 
  and $f_0$ by $+(e+1)/2$.

  Refer to Figure~\ref{fig:delta1}. Since $L_1$ and $L_\infty$ have compatible 
  orientations, Lemma~\ref{lem:orient_delta} shows that the map $f_1$ shifts 
  the $\delta$-grading by $+1/2$.

  \begin{figure}
    \resizebox{0.50\textwidth}{!}{\input{fig_delta1.pstex_t}}
    \caption{A straightforward application of Lemma~\ref{lem:orient_delta} 
      gives the $\delta$-grading shift of $f_1$.}
    \label{fig:delta1}
  \end{figure}
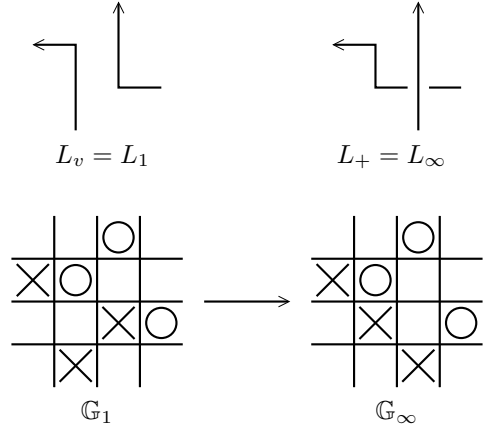

  Let us now focus on $f_\infty$. Let $\gen{x} \in \gen{S} (\grid_\infty)$ be a 
  generator, and suppose $\gen{y} \in \gen{S} (\grid_0)$ appears in $f_\infty 
  (\gen{x})$. We proceed according to whether the two strands of the link $L_+$ 
  meeting at the crossing $c_0$ belong to different components, or to the same 
  component, of $L_+$, as follows.

  \begin{figure}
    \resizebox{0.818\textwidth}{!}{\input{fig_delta2_diff.pstex_t}}
    \caption{Computation of the $\delta$-grading shift of $f_\infty$, when the 
      two strands meeting at $c_0$ belong to different components. The 
      generators $\gen{x}$ and $\gen{y}$ are not shown, as their positions do 
      not affect the argument.}
    \label{fig:delta2_diff}
  \end{figure}
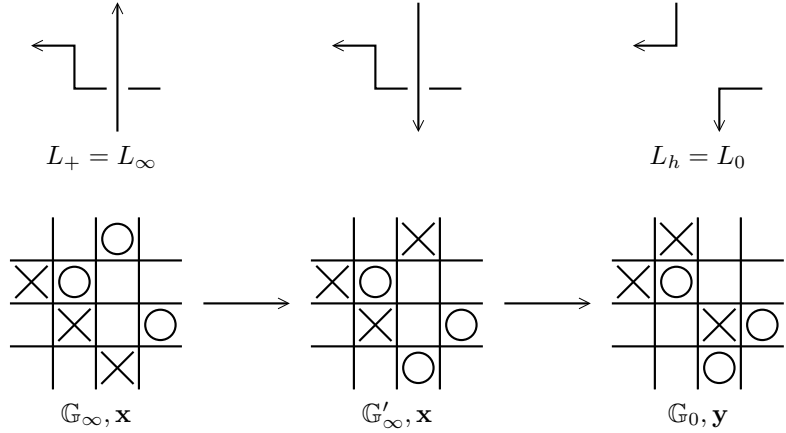

  Suppose the strands belong to different components; see 
  Figure~\ref{fig:delta2_diff}. Starting from $\grid_\infty$, we can reverse 
  the orientation of one of the two components to obtain a new diagram 
  $\grid_\infty'$; then we can write
  \[
    \delta_{\grid_0} (\gen{y}) - \delta_{\grid_\infty} (\gen{x}) = 
    (\delta_{\grid_\infty'} (\gen{x}) - \delta_{\grid_\infty} (\gen{x})) + 
    (\delta_{\grid_0} (\gen{y}) - \delta_{\grid_\infty'} (\gen{x})).
  \]
  By Lemma~\ref{lem:delta}~(3), we have $\delta_{\grid_\infty'} (\gen{x}) - 
  \delta_{\grid_\infty} (\gen{x}) = -\epsilon/2$, where $\epsilon$ is the 
  difference between the number of negative crossings in $\grid_\infty'$ and 
  the number of such crossings in $\grid_\infty$ (which represents $L_+$). Now 
  observe that $\grid_\infty'$ and $\grid_0$ have compatible orientations; 
  also, $\grid_\infty'$ has one more negative crossing, near $c_0$, than does 
  $\grid_0$ (which represents $L_h$).  Therefore, $\epsilon = e + 1$.  
  Furthermore, the fact that $\grid_\infty'$ and $\grid_0$ have compatible 
  orientations implies that Lemma~\ref{lem:orient_delta} can be applied, so 
  $\delta_{\grid_0} (\gen{y}) - \delta_{\grid_\infty'} (\gen{x}) = +1/2$. Thus,
  \[
    \delta_{\grid_0} (\gen{y}) - \delta_{\grid_\infty} (\gen{x}) = -(e+1)/2 + 
    1/2 = -e/2.
  \]

  \begin{figure}
    \resizebox{\textwidth}{!}{\input{fig_delta2_same.pstex_t}}
    \caption{Computation of the $\delta$-grading shift of $f_\infty$, when the 
      two strands meeting at $c_0$ belong to the same component. The 
      computation depends on the position of $x$, the $\beta$-component of 
      $\gen{x}$; the special cases are indicated by the solid point, the hollow 
      point, and the hollow square.}
    \label{fig:delta2_same}
  \end{figure}
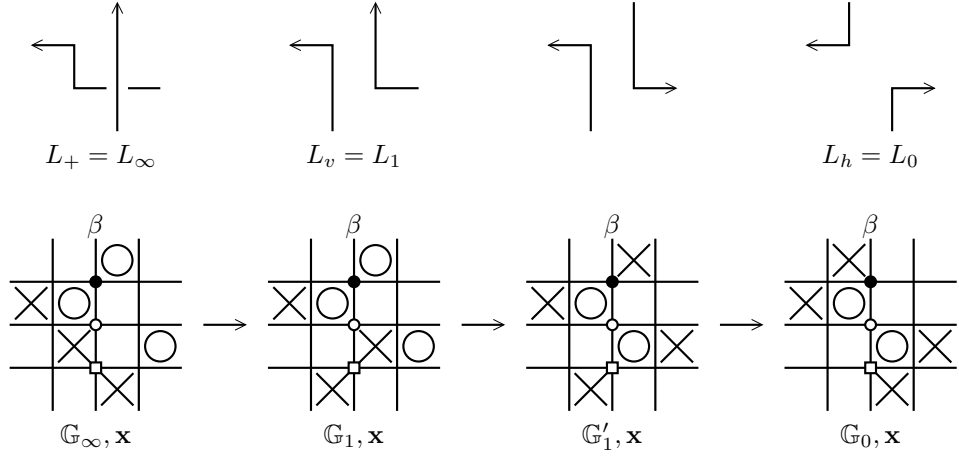

  Suppose now the strands belong to the same component. As in the proof of 
  Lemma~\ref{lem:orient_delta}, we write
  \[
    \delta_{\grid_0} (\gen{y}) - \delta_{\grid_\infty} (\gen{x}) = 
    (\delta_{\grid_0} (\gen{x}) - \delta_{\grid_\infty} (\gen{x})) + 
    (\delta_{\grid_0} (\gen{y}) - \delta_{\grid_0} (\gen{x})).
  \]
  Note that this is different from the equation we use when the strands belong 
  to different components. We first compute $\delta_{\grid_0} (\gen{x}) - 
  \delta_{\grid_\infty} (\gen{x})$; see Figure~\ref{fig:delta2_same}. Since the 
  two strands in $L_+$ belong to the same component, the two strands in $L_v$ 
  must belong to different components. Therefore, we can reverse the 
  orientation of one of these components in $\grid_1$ (which represents $L_v$) 
  to obtain a new diagram $\grid_1'$; then
  \[
    \delta_{\grid_0} (\gen{x}) - \delta_{\grid_\infty} (\gen{x}) = 
    (\delta_{\grid_1} (\gen{x}) - \delta_{\grid_\infty} (\gen{x})) + 
    (\delta_{\grid_1'} (\gen{x}) - \delta_{\grid_1} (\gen{x})) + 
    (\delta_{\grid_0} (\gen{x}) - \delta_{\grid_1'} (\gen{x})).
  \]
  Let $x$ be the $\beta$-component of $\gen{x}$. We can now compute, by 
  Lemma~\ref{lem:delta}~(2),
  \[
    \delta_{\grid_1} (\gen{x}) - \delta_{\grid_\infty} (\gen{x}) =
    \begin{cases}
      +1/2 & \text{if } x \text{ lies at the hollow square in 
        Figure~\ref{fig:delta2_same};}\\
      -1/2 & \text{otherwise.}
    \end{cases}
  \]
  Now, by Lemma~\ref{lem:delta}~(3), $\delta_{\grid_1'} (\gen{x}) - 
  \delta_{\grid_1} (\gen{x}) = -\epsilon/2$, where $\epsilon$ is the difference 
  between the number of negative crossings in $\grid_1'$ and the number of such 
  crossings in $\grid_1$. Observe that $\grid_1$ and $\grid_\infty$ (which 
  represents $L_+$) have compatible orientations and hence the same number of 
  negative crossings. Similarly, $\grid_1'$ and $\grid_0$ have compatible 
  orientations and hence the same number of negative crossings. Therefore, we 
  see that $\epsilon = e$. Next, again by Lemma~\ref{lem:delta}~(2), we have
  \[
    \delta_{\grid_0} (\gen{x}) - \delta_{\grid_1'} (\gen{x}) =
    \begin{cases}
      -1/2 & \text{if } x \text{ lies at the solid point, hollow point,}\\
      & \text{\quad or hollow square in Figure~\ref{fig:delta2_same};}\\
      +1/2 & \text{otherwise.}
    \end{cases}
  \]
  Combining these calculations, we conclude that
  \[
    \delta_{\grid_0} (\gen{x}) - \delta_{\grid_\infty} (\gen{x}) =
    \begin{cases}
      -(e+2)/2 & \text{if } x \text{ lies at the solid point or hollow point}\\
      & \text{\quad in Figure~\ref{fig:delta2_same};}\\
      -e/2 & \text{otherwise.}
    \end{cases}
  \]

  We now return to computing $\delta_{\grid_0} (\gen{y}) - 
  \delta_{\grid_\infty} (\gen{x})$.
  
  Suppose there is an empty pentagon $p$ from $\gen{x}$ to $\gen{y}$; in the 
  framework described in this section, it can in fact be viewed as an empty 
  rectangle $r$ from $\gen{x}$ to $\gen{y}$ in $\grid_0$. By an argument 
  similar to that in the proof of Lemma~\ref{lem:orient_delta}, one of two 
  scenarios must be true:
  \begin{enumerate}
    \item If the point $x$ lies at the solid point or at the hollow point in 
      Figure~\ref{fig:delta2_same}, then $\Int (r) \cap (\markersO \cup 
      \markers) = \eset$.  By Lemma~\ref{lem:delta}~(1), $\delta_{\grid_0} 
      (\gen{y}) - \delta_{\grid_0} (\gen{x}) = 1$.
    \item If the point $x$ lies elsewhere, then $\Int (r) \cap (\markersO \cup 
      \markers)$ contains exactly one point. By Lemma~\ref{lem:delta}~(1), 
      $\delta_{\grid_0} (\gen{y}) - \delta_{\grid_0} (\gen{x}) = 1-1 = 0$.  
  \end{enumerate}
  Combining with our earlier calculation, we get that $\delta_{\grid_0} 
  (\gen{y}) - \delta_{\grid_\infty} (\gen{x}) = -e/2$.

  Suppose there is a triangle $p$ from $\gen{x}$ to $\gen{y}$; then, in the 
  framework of this section, $\gen{x} = \gen{y}$, and so $\delta_{\grid_0} 
  (\gen{y}) - \delta_{\grid_0} (\gen{x}) = 0$.  Since the boundary of $p$ 
  contains $v_\infty$ (cf.\ Figure~\ref{fig:one_grid}), we see that $\Int (p) 
  \subset \Int (t_\infty) \cup \Int (t_0)$. This implies that $x$ cannot lie at 
  the solid point or at the hollow point in Figure~\ref{fig:delta2_same}.  
  Therefore, by our earlier calculation, again we have $\delta_{\grid_0} 
  (\gen{y}) - \delta_{\grid_\infty} (\gen{x}) = -e/2$.

  Finally, a calculation analogous to that for $f_\infty$ can be done for 
  $f_0$, and we obtain, in this case, that the shift in the $\delta$-grading is 
  $+(e+1)/2$.

  The case when $L_\infty = L_-$ has a negative crossing at $c_0$ is similar.
\end{proof}

\begin{rmk}
  The proof above only shows that the maps $f_k$ are graded chain maps in the 
  exact sequence.  To show that $\GCt (\grid_k)$ is \emph{graded} 
  quasi-isomorphic to the mapping cone of $f_{k+1}$ (cf.\ 
  Section~\ref{sec:iterate}), we would have to show analogous statements for 
  the chain homotopies $\phi_k$ also. The proof of such statements are omitted 
  here, but are completely analogous to the proof of 
  Proposition~\ref{prop:delta_main}.
\end{rmk}

\begin{proof}[Proof of Theorem~\ref{thm:delta_sigma}]
  This is just a restatement of Proposition~\ref{prop:delta_main}, taking into 
  account the result of Lemma~\ref{lem:sigma}.
\end{proof}

\begin{proof}[Proof of Corollary~\ref{thm:thinZ}]
  By definition, every quasi-alternating link has non-zero determinant. It can 
  easily be checked that the unknot is homologically $\sigma$-thin. If $L_0$ 
  and $L_1$ are resolutions of $L_\infty$ that are quasi-alternating, then by 
  induction, $L_0$ and $L_1$ are homologically $\sigma$-thin. The exact 
  sequence in Proposition~\ref{thm:delta_sigma} collapses into short exact 
  sequences, all but one of which are zero. Thus $L_\infty$ is also 
  homologically $\sigma$-thin.
\end{proof}

%% file: fig_links2.pstex_t
\begin{picture}(0,0)%
\includegraphics{fig_links2.pstex}%
\end{picture}%
\setlength{\unitlength}{1827sp}%
\begingroup\makeatletter\ifx\SetFigFont\undefined%
\gdef\SetFigFont#1#2#3#4#5{%
  \reset@font\fontsize{#1}{#2pt}%
  \fontfamily{#3}\fontseries{#4}\fontshape{#5}%
  \selectfont}%
\fi\endgroup%
\begin{picture}(12966,2343)(-2732,-2671)
\put(-1799,-2611){\makebox(0,0)[b]{\smash{{\SetFigFont{10}{12.0}{\familydefault}{\mddefault}{\updefault}{\color[rgb]{0,0,0}$L_+$}%
}}}}
\put(2101,-2611){\makebox(0,0)[b]{\smash{{\SetFigFont{10}{12.0}{\familydefault}{\mddefault}{\updefault}{\color[rgb]{0,0,0}$L_-$}%
}}}}
\put(5701,-2611){\makebox(0,0)[b]{\smash{{\SetFigFont{10}{12.0}{\familydefault}{\mddefault}{\updefault}{\color[rgb]{0,0,0}$L_v$}%
}}}}
\put(9301,-2611){\makebox(0,0)[b]{\smash{{\SetFigFont{10}{12.0}{\familydefault}{\mddefault}{\updefault}{\color[rgb]{0,0,0}$L_h$}%
}}}}
\end{picture}%

%% file: fig_markers.pstex_t
\begin{picture}(0,0)%
\includegraphics{fig_markers.pstex}%
\end{picture}%
\setlength{\unitlength}{2191sp}%
\begingroup\makeatletter\ifx\SetFigFont\undefined%
\gdef\SetFigFont#1#2#3#4#5{%
  \reset@font\fontsize{#1}{#2pt}%
  \fontfamily{#3}\fontseries{#4}\fontshape{#5}%
  \selectfont}%
\fi\endgroup%
\begin{picture}(5466,3552)(2068,-7180)
\put(3001,-7111){\makebox(0,0)[b]{\smash{{\SetFigFont{10}{12.0}{\rmdefault}{\mddefault}{\updefault}{\color[rgb]{0,0,0}$\grid, \gen{x}$}%
}}}}
\put(6601,-7111){\makebox(0,0)[b]{\smash{{\SetFigFont{10}{12.0}{\rmdefault}{\mddefault}{\updefault}{\color[rgb]{0,0,0}$\grid', \gen{x}$}%
}}}}
\end{picture}%

%% file: fig_orient_delta.pstex_t
\begin{picture}(0,0)%
\includegraphics{fig_orient_delta.pstex}%
\end{picture}%
\setlength{\unitlength}{2131sp}%
\begingroup\makeatletter\ifx\SetFigFont\undefined%
\gdef\SetFigFont#1#2#3#4#5{%
  \reset@font\fontsize{#1}{#2pt}%
  \fontfamily{#3}\fontseries{#4}\fontshape{#5}%
  \selectfont}%
\fi\endgroup%
\begin{picture}(9066,8517)(2068,-10180)
\put(4501,-9136){\makebox(0,0)[b]{\smash{{\SetFigFont{10}{12.0}{\familydefault}{\mddefault}{\updefault}{\color[rgb]{0,0,0}$O_2$}%
}}}}
\put(5101,-8536){\makebox(0,0)[b]{\smash{{\SetFigFont{10}{12.0}{\familydefault}{\mddefault}{\updefault}{\color[rgb]{0,0,0}$O_1$}%
}}}}
\put(8101,-8536){\makebox(0,0)[b]{\smash{{\SetFigFont{10}{12.0}{\familydefault}{\mddefault}{\updefault}{\color[rgb]{0,0,0}$O_1'$}%
}}}}
\put(8701,-9136){\makebox(0,0)[b]{\smash{{\SetFigFont{10}{12.0}{\familydefault}{\mddefault}{\updefault}{\color[rgb]{0,0,0}$O_2'$}%
}}}}
\put(4801,-10111){\makebox(0,0)[b]{\smash{{\SetFigFont{10}{12.0}{\familydefault}{\mddefault}{\updefault}{\color[rgb]{0,0,0}$\grid_k, \gen{x}$}%
}}}}
\put(4801,-7786){\makebox(0,0)[b]{\smash{{\SetFigFont{10}{12.0}{\familydefault}{\mddefault}{\updefault}{\color[rgb]{0,0,0}$\beta$}%
}}}}
\put(8401,-7786){\makebox(0,0)[b]{\smash{{\SetFigFont{10}{12.0}{\familydefault}{\mddefault}{\updefault}{\color[rgb]{0,0,0}$\beta$}%
}}}}
\put(6601,-7111){\makebox(0,0)[b]{\smash{{\SetFigFont{10}{12.0}{\familydefault}{\mddefault}{\updefault}{\color[rgb]{0,0,0}$\grid_{k+1}, \gen{x}$}%
}}}}
\put(10201,-7111){\makebox(0,0)[b]{\smash{{\SetFigFont{10}{12.0}{\familydefault}{\mddefault}{\updefault}{\color[rgb]{0,0,0}$\grid_{k+1}, \gen{y}$}%
}}}}
\put(3001,-7111){\makebox(0,0)[b]{\smash{{\SetFigFont{10}{12.0}{\familydefault}{\mddefault}{\updefault}{\color[rgb]{0,0,0}$\grid_k, \gen{x}$}%
}}}}
\put(2701,-6136){\makebox(0,0)[b]{\smash{{\SetFigFont{10}{12.0}{\familydefault}{\mddefault}{\updefault}{\color[rgb]{0,0,0}$O_2$}%
}}}}
\put(6901,-6136){\makebox(0,0)[b]{\smash{{\SetFigFont{10}{12.0}{\familydefault}{\mddefault}{\updefault}{\color[rgb]{0,0,0}$O_2'$}%
}}}}
\put(9901,-5536){\makebox(0,0)[b]{\smash{{\SetFigFont{10}{12.0}{\familydefault}{\mddefault}{\updefault}{\color[rgb]{0,0,0}$O_1'$}%
}}}}
\put(10501,-6136){\makebox(0,0)[b]{\smash{{\SetFigFont{10}{12.0}{\familydefault}{\mddefault}{\updefault}{\color[rgb]{0,0,0}$O_2'$}%
}}}}
\put(3301,-5536){\makebox(0,0)[b]{\smash{{\SetFigFont{10}{12.0}{\familydefault}{\mddefault}{\updefault}{\color[rgb]{0,0,0}$O_1$}%
}}}}
\put(6301,-5536){\makebox(0,0)[b]{\smash{{\SetFigFont{10}{12.0}{\familydefault}{\mddefault}{\updefault}{\color[rgb]{0,0,0}$O_1'$}%
}}}}
\put(6601,-4786){\makebox(0,0)[b]{\smash{{\SetFigFont{10}{12.0}{\familydefault}{\mddefault}{\updefault}{\color[rgb]{0,0,0}$\beta$}%
}}}}
\put(6751,-5086){\makebox(0,0)[b]{\smash{{\SetFigFont{10}{12.0}{\familydefault}{\mddefault}{\updefault}{\color[rgb]{0,0,0}$x$}%
}}}}
\put(3001,-4786){\makebox(0,0)[b]{\smash{{\SetFigFont{10}{12.0}{\familydefault}{\mddefault}{\updefault}{\color[rgb]{0,0,0}$\beta$}%
}}}}
\put(10201,-4786){\makebox(0,0)[b]{\smash{{\SetFigFont{10}{12.0}{\familydefault}{\mddefault}{\updefault}{\color[rgb]{0,0,0}$\beta$}%
}}}}
\put(6601,-4111){\makebox(0,0)[b]{\smash{{\SetFigFont{10}{12.0}{\familydefault}{\mddefault}{\updefault}{\color[rgb]{0,0,0}$\grid_{k+1}, \gen{x}$}%
}}}}
\put(10201,-4111){\makebox(0,0)[b]{\smash{{\SetFigFont{10}{12.0}{\familydefault}{\mddefault}{\updefault}{\color[rgb]{0,0,0}$\grid_{k+1}, \gen{y}$}%
}}}}
\put(3001,-4111){\makebox(0,0)[b]{\smash{{\SetFigFont{10}{12.0}{\familydefault}{\mddefault}{\updefault}{\color[rgb]{0,0,0}$\grid_k, \gen{x}$}%
}}}}
\put(2701,-3136){\makebox(0,0)[b]{\smash{{\SetFigFont{10}{12.0}{\familydefault}{\mddefault}{\updefault}{\color[rgb]{0,0,0}$O_2$}%
}}}}
\put(6901,-3136){\makebox(0,0)[b]{\smash{{\SetFigFont{10}{12.0}{\familydefault}{\mddefault}{\updefault}{\color[rgb]{0,0,0}$O_2'$}%
}}}}
\put(9901,-2536){\makebox(0,0)[b]{\smash{{\SetFigFont{10}{12.0}{\familydefault}{\mddefault}{\updefault}{\color[rgb]{0,0,0}$O_1'$}%
}}}}
\put(10501,-3136){\makebox(0,0)[b]{\smash{{\SetFigFont{10}{12.0}{\familydefault}{\mddefault}{\updefault}{\color[rgb]{0,0,0}$O_2'$}%
}}}}
\put(3301,-2536){\makebox(0,0)[b]{\smash{{\SetFigFont{10}{12.0}{\familydefault}{\mddefault}{\updefault}{\color[rgb]{0,0,0}$O_1$}%
}}}}
\put(6301,-2536){\makebox(0,0)[b]{\smash{{\SetFigFont{10}{12.0}{\familydefault}{\mddefault}{\updefault}{\color[rgb]{0,0,0}$O_1'$}%
}}}}
\put(6601,-1786){\makebox(0,0)[b]{\smash{{\SetFigFont{10}{12.0}{\familydefault}{\mddefault}{\updefault}{\color[rgb]{0,0,0}$\beta$}%
}}}}
\put(3001,-1786){\makebox(0,0)[b]{\smash{{\SetFigFont{10}{12.0}{\familydefault}{\mddefault}{\updefault}{\color[rgb]{0,0,0}$\beta$}%
}}}}
\put(10201,-1786){\makebox(0,0)[b]{\smash{{\SetFigFont{10}{12.0}{\familydefault}{\mddefault}{\updefault}{\color[rgb]{0,0,0}$\beta$}%
}}}}
\put(3151,-5086){\makebox(0,0)[b]{\smash{{\SetFigFont{10}{12.0}{\familydefault}{\mddefault}{\updefault}{\color[rgb]{0,0,0}$x$}%
}}}}
\put(4951,-8086){\makebox(0,0)[b]{\smash{{\SetFigFont{10}{12.0}{\familydefault}{\mddefault}{\updefault}{\color[rgb]{0,0,0}$x$}%
}}}}
\put(2851,-2686){\makebox(0,0)[b]{\smash{{\SetFigFont{10}{12.0}{\familydefault}{\mddefault}{\updefault}{\color[rgb]{0,0,0}$x$}%
}}}}
\put(6301,-5986){\makebox(0,0)[b]{\smash{{\SetFigFont{10}{12.0}{\familydefault}{\mddefault}{\updefault}{\color[rgb]{0,0,0}$r$}%
}}}}
\put(9901,-5986){\makebox(0,0)[b]{\smash{{\SetFigFont{10}{12.0}{\familydefault}{\mddefault}{\updefault}{\color[rgb]{0,0,0}$r$}%
}}}}
\put(6901,-2536){\makebox(0,0)[b]{\smash{{\SetFigFont{10}{12.0}{\familydefault}{\mddefault}{\updefault}{\color[rgb]{0,0,0}$r$}%
}}}}
\put(10501,-2536){\makebox(0,0)[b]{\smash{{\SetFigFont{10}{12.0}{\familydefault}{\mddefault}{\updefault}{\color[rgb]{0,0,0}$r$}%
}}}}
\put(6451,-2986){\makebox(0,0)[b]{\smash{{\SetFigFont{10}{12.0}{\familydefault}{\mddefault}{\updefault}{\color[rgb]{0,0,0}$x$}%
}}}}
\put(8401,-10111){\makebox(0,0)[b]{\smash{{\SetFigFont{10}{12.0}{\familydefault}{\mddefault}{\updefault}{\color[rgb]{0,0,0}$\grid_{k+1}, \gen{x} = \gen{y}$}%
}}}}
\end{picture}%

%% file: fig_delta1.pstex_t
\begin{picture}(0,0)%
\includegraphics{fig_delta1.pstex}%
\end{picture}%
\setlength{\unitlength}{1796sp}%
\begingroup\makeatletter\ifx\SetFigFont\undefined%
\gdef\SetFigFont#1#2#3#4#5{%
  \reset@font\fontsize{#1}{#2pt}%
  \fontfamily{#3}\fontseries{#4}\fontshape{#5}%
  \selectfont}%
\fi\endgroup%
\begin{picture}(6666,5947)(-2432,-6875)
\put(-1124,-3211){\makebox(0,0)[b]{\smash{{\SetFigFont{10}{12.0}{\familydefault}{\mddefault}{\updefault}{\color[rgb]{0,0,0}$L_v = L_1$}%
}}}}
\put(2926,-3211){\makebox(0,0)[b]{\smash{{\SetFigFont{10}{12.0}{\familydefault}{\mddefault}{\updefault}{\color[rgb]{0,0,0}$L_+ = L_\infty$}%
}}}}
\put(3001,-6811){\makebox(0,0)[b]{\smash{{\SetFigFont{10}{12.0}{\familydefault}{\mddefault}{\updefault}{\color[rgb]{0,0,0}$\grid_\infty$}%
}}}}
\put(-1199,-6811){\makebox(0,0)[b]{\smash{{\SetFigFont{10}{12.0}{\familydefault}{\mddefault}{\updefault}{\color[rgb]{0,0,0}$\grid_1$}%
}}}}
\end{picture}%

%% file: fig_delta2_diff.pstex_t
\begin{picture}(0,0)%
\includegraphics{fig_delta2_diff.pstex}%
\end{picture}%
\setlength{\unitlength}{1776sp}%
\begingroup\makeatletter\ifx\SetFigFont\undefined%
\gdef\SetFigFont#1#2#3#4#5{%
  \reset@font\fontsize{#1}{#2pt}%
  \fontfamily{#3}\fontseries{#4}\fontshape{#5}%
  \selectfont}%
\fi\endgroup%
\begin{picture}(10866,5952)(-2432,-6880)
\put(-1124,-3211){\makebox(0,0)[b]{\smash{{\SetFigFont{10}{12.0}{\familydefault}{\mddefault}{\updefault}{\color[rgb]{0,0,0}$L_+ = L_\infty$}%
}}}}
\put(-1199,-6811){\makebox(0,0)[b]{\smash{{\SetFigFont{10}{12.0}{\familydefault}{\mddefault}{\updefault}{\color[rgb]{0,0,0}$\grid_\infty, \gen{x}$}%
}}}}
\put(3001,-6811){\makebox(0,0)[b]{\smash{{\SetFigFont{10}{12.0}{\familydefault}{\mddefault}{\updefault}{\color[rgb]{0,0,0}$\grid_\infty', \gen{x}$}%
}}}}
\put(7201,-6811){\makebox(0,0)[b]{\smash{{\SetFigFont{10}{12.0}{\familydefault}{\mddefault}{\updefault}{\color[rgb]{0,0,0}$\grid_0, \gen{y}$}%
}}}}
\put(7201,-3211){\makebox(0,0)[b]{\smash{{\SetFigFont{10}{12.0}{\familydefault}{\mddefault}{\updefault}{\color[rgb]{0,0,0}$L_h = L_0$}%
}}}}
\end{picture}%

%% file: fig_delta2_same.pstex_t
\begin{picture}(0,0)%
\includegraphics{fig_delta2_same.pstex}%
\end{picture}%
\setlength{\unitlength}{1776sp}%
\begingroup\makeatletter\ifx\SetFigFont\undefined%
\gdef\SetFigFont#1#2#3#4#5{%
  \reset@font\fontsize{#1}{#2pt}%
  \fontfamily{#3}\fontseries{#4}\fontshape{#5}%
  \selectfont}%
\fi\endgroup%
\begin{picture}(13266,6252)(-2432,-6880)
\put(-1199,-6811){\makebox(0,0)[b]{\smash{{\SetFigFont{10}{12.0}{\familydefault}{\mddefault}{\updefault}{\color[rgb]{0,0,0}$\grid_\infty, \gen{x}$}%
}}}}
\put(6001,-6811){\makebox(0,0)[b]{\smash{{\SetFigFont{10}{12.0}{\familydefault}{\mddefault}{\updefault}{\color[rgb]{0,0,0}$\grid_1', \gen{x}$}%
}}}}
\put(9601,-6811){\makebox(0,0)[b]{\smash{{\SetFigFont{10}{12.0}{\familydefault}{\mddefault}{\updefault}{\color[rgb]{0,0,0}$\grid_0, \gen{x}$}%
}}}}
\put(2401,-6811){\makebox(0,0)[b]{\smash{{\SetFigFont{10}{12.0}{\familydefault}{\mddefault}{\updefault}{\color[rgb]{0,0,0}$\grid_1, \gen{x}$}%
}}}}
\put(-1199,-3886){\makebox(0,0)[b]{\smash{{\SetFigFont{10}{12.0}{\familydefault}{\mddefault}{\updefault}{\color[rgb]{0,0,0}$\beta$}%
}}}}
\put(2401,-3886){\makebox(0,0)[b]{\smash{{\SetFigFont{10}{12.0}{\familydefault}{\mddefault}{\updefault}{\color[rgb]{0,0,0}$\beta$}%
}}}}
\put(6001,-3886){\makebox(0,0)[b]{\smash{{\SetFigFont{10}{12.0}{\familydefault}{\mddefault}{\updefault}{\color[rgb]{0,0,0}$\beta$}%
}}}}
\put(9601,-3886){\makebox(0,0)[b]{\smash{{\SetFigFont{10}{12.0}{\familydefault}{\mddefault}{\updefault}{\color[rgb]{0,0,0}$\beta$}%
}}}}
\put(-1124,-2911){\makebox(0,0)[b]{\smash{{\SetFigFont{10}{12.0}{\familydefault}{\mddefault}{\updefault}{\color[rgb]{0,0,0}$L_+ = L_\infty$}%
}}}}
\put(9601,-2911){\makebox(0,0)[b]{\smash{{\SetFigFont{10}{12.0}{\familydefault}{\mddefault}{\updefault}{\color[rgb]{0,0,0}$L_h = L_0$}%
}}}}
\put(2401,-2911){\makebox(0,0)[b]{\smash{{\SetFigFont{10}{12.0}{\familydefault}{\mddefault}{\updefault}{\color[rgb]{0,0,0}$L_v = L_1$}%
}}}}
\end{picture}%